\documentclass[graybox]{svmult}

\usepackage{mathptmx}       
\usepackage{helvet}         
\usepackage{courier}        
\usepackage{type1cm}

\usepackage{color}
\usepackage{graphicx,subfigure,amsfonts,amsmath,amsxtra}

\usepackage{amssymb,mathrsfs}
\usepackage[mathscr]{eucal}
\usepackage{pstricks}

\usepackage[sort,comma]{natbib}

\allowdisplaybreaks

\makeatletter
\newcommand{\pushright}[1]{\ifmeasuring@#1\else\omit\hfill$\displaystyle#1$\fi\ignorespaces}
\newcommand{\pushleft}[1]{\ifmeasuring@#1\else\omit$\displaystyle#1$\hfill\fi\ignorespaces}
\makeatother

\def\qed{\hfill{$\Box$}}
\def\F{{\mathcal F}}
 
\def\ss{{\mathbb  S}}
\def\lan{\big\langle}
\def\ran{\big\rangle}
\newcommand{\LL}{\mathcal L}
\newcommand{\R}{\mathbb R}
\newcommand{\A}{\mathcal A}
\newcommand{\B}{\mathfrak B}
\renewcommand{\d}{\mathrm{d}}

\newcommand{\la}{\lambda}
\newcommand{\La}{\Lambda}

\newcommand{\e}{\varepsilon}
\newcommand{\m}{{\mathfrak m}}
\newcommand{\al }{\alpha}

\newcommand{\sg}{\sigma}

 \newcommand{\tr}{\text {tr}}
 
 \newcommand{\sgn}{\text{sgn}}
\renewcommand{\P}{\mathbb P}

\renewcommand{\E}{{\mathbb E}}

\newcommand{\N}{\mathbb N}
\newcommand{\disp}{\displaystyle}

\newcommand{\wdt}{\widetilde}

\newtheorem{thm}{Theorem}[section]

\newtheorem{lem}[thm]{Lemma}

\newtheorem{Assumption}[thm]{Assumption}
\newtheorem{rem}[thm]{Remark}

\newtheorem{exm}[thm]{Example}

\newcommand{\set}[1]{\left\{#1\right\}}
\newcommand{\abs}[1]{\left\vert #1\right\vert}

\def\george#1 {\fbox {\footnote {\ }}\ \footnotetext {From George: #1}}
\def\fubao#1 {\fbox {\footnote {\ }}\ \footnotetext {From Fubao: #1}}
\def\chao#1 {\fbox {\footnote {\ }}\ \footnotetext {From Chao: #1}}
\newcommand{\one}{\mathbf 1}

\begin{document}

\title*{Regime-Switching Jump Diffusions with Non-Lipschitz  Coefficients and Countably Many Switching States: Existence and Uniqueness, Feller, and Strong Feller Properties\thanks{The study was initiated during a workshop held at IMA, University of Minnesota. The support of IMA with funding provided by the National Science Foundation is acknowledged. The research was also supported in part
by the National Natural Science Foundation of China under Grant No. 11671034,  the US Army Research Office, and the Simons Foundation  Collaboration Grant (No. 523736).}
}
\titlerunning{Regime-Switching Jump Diffusions}

\author{Fubao Xi
\and George Yin
\and Chao Zhu}

\institute{F. Xi \at  School of Mathematics and Statistics, Beijing Institute of Technology, Beijing 100081, China, \email{xifb@bit.edu.cn}
\and G. Yin \at Department of Mathematics,
Wayne State University, Detroit, MI 48202, \email{gyin@math.wayne.edu} \and C. Zhu \at Department of Mathematical
Sciences, University of Wisconsin-Milwaukee, Milwaukee,
WI 53201, \email{zhu@uwm.edu}}

\maketitle

\abstract*{
This work focuses on a class of regime-switching jump diffusion processes, which
is a two component Markov processes $(X(t),\Lambda(t))$, where $\Lambda(t)$ is a component representing discrete events taking values in a countably infinite set.
Considering the corresponding stochastic differential equations,
 our main focus is  on treating
 those with non-Lipschitz coefficients.
 We first show that there exists a unique strong solution to the corresponding stochastic differential equation. Then Feller   and strong Feller properties are investigated.
}

\abstract{
This work focuses on a class of regime-switching jump diffusion processes, which
is a two component Markov processes $(X(t),\Lambda(t))$, where $\Lambda(t)$ is a component representing discrete events taking values in a countably infinite set.
Considering the corresponding stochastic differential equations,
 our main focus is  on treating
 those with non-Lipschitz coefficients.
 We first show that there exists a unique strong solution to the corresponding stochastic differential equation. Then Feller   and strong Feller properties are investigated.
\newline \vskip 0.02 in
\noindent {\bf Keywords.} Regime-switching jump diffusion, non-Lipschitz condition, Feller property, strong Feller property.
\newline \vskip 0.02 in
\noindent {\bf Mathematics Subject Classification.} 60J27, 60J60, 60J75, 60G51.
}

\section{Introduction}\label{sect:intro}
In the past decade, much attention has been devoted to a class of hybrid systems, namely, regime-switching diffusions.
Roughly, such processes can be considered as a two component process
$(X(t),\Lambda(t))$, an analog  (or continuous
 state) component $X(t)$ and a switching (or discrete event)
  process $\Lambda(t)$.
Some of the representative works can be found in \cite{MaoY} and \cite{YZ-10}.
 The former dealt with
 regime-switching diffusions in which the switching process is a continuous-time Markov chain independent of the Brownian motion,
 whereas the latter treated  processes in which the switching component depends on the continuous-state component.
 It has been found that the discrete event process,
 taking values in a finite or countable set,
 can be used to delineate, for example, random environment or other random factors that are not represented in the usual diffusion formulation. Seemingly similar to the diffusion processes,
  in fact,
  regime-switching diffusions have very different behavior compared to the usual diffusion processes.  For example, it has been demonstrated in  \cite{LawleyMR-14,YinZW-12} that two stable (resp., unstable) ordinary differential equations can be coupled to produce an unstable (resp., stable) regime-switching system.
  The consideration of regime-switching diffusions has substantially enlarged the applicability of stochastic processes for a wide variety of problems ranging from network systems, multi-agent systems, ecological and biological applications,  financial engineering, risk management, etc.

 Continuing on the effort of studying  regime-switching diffusions,
 \cite{CCTY18a} obtained maximum principle and Harnack inequalities for switching jump diffusions using mainly probabilistic arguments, and \cite{CCTY18b} proceeded further to obtain recurrence and ergodicity of switching jump diffusions. In another direction,  \cite{XiZ-17} dealt with regime-switching jump diffusions with countable number of switching values. \cite{NgurenY-16} considered
  switching diffusions in which the switching process depends on the past information of the continuous state and takes values in a countable state space; the corresponding recurrence and ergodicity was considered in \cite{NguyenY-18}.

A standing assumption in the aforementioned references is that the coefficients of the associated stochastic differential equations are (locally) Lipschitz. While it is a convenient assumption, it is rather restrictive in many applications. For example,   the diffusion coefficients in the   Feller branching diffusion  and the Cox-Ingersoll-Ross model  are only H\"older continuous. We refer to Chapters 12 and 13 of \cite{Klebaner-05} for an introduction to these models.
Motivated by these considerations, there has been much efforts devoted to the study of stochastic differential equations with non-Lipschitz coefficients. An incomplete list includes \cite{YamaW-71,LiPu-12,FangZ-05,LiMyt-11,Bass-03}, among many others.

 While there are many works on diffusions and jump diffusions with non-Lipschitz coefficients, the related research on regime-switching jump diffusions is relatively scarce. This work aims to investigate regime-switching jump diffusion processes with non-Lipschitz coefficients. More precisely, the purpose of this paper is two-fold: (i) to establish the strong existence and uniqueness result for stochastic differential equations associated with  regime-switching jump diffusions, in which the coefficients are non-Lipschitz and the switching component has countably  many states;  and (ii) to derive sufficient conditions for Feller and strong Feller properties. Our focus is devoted to establishing non-Lipschitz sufficient conditions for the aforementioned properties.


 The rest of the paper is arranged as follows. Examining the associated stochastic differential equations,
 we begin to obtain the existence and uniqueness of the solution of the stochastic differential equations in Section \ref{sect-str-formulation}. Then Section \ref{sect-Feller} proceeds with the study of Feller properties.
 Section \ref{sect-str-Feller} further extends the study to treat strong Feller properties.

\section{Strong Solution: Existence and Uniqueness}\label{sect-str-formulation}

  We work with $(U, \mathfrak U)$
   a measurable space, $\nu$ a $\sigma$-finite measure on $U$, and
$\ss:=\{  1, 2, \dots \}$.
Assume that  $d \ge 1 $ is a positive integer,
 $b: \R^{d}\times \ss \mapsto \R^{d}$, $\sigma:\R^{d}\times \ss \mapsto \R^{d\times d}$, and $c: \R^{d}\times \ss \times U \mapsto \R^{d}$ be Borel measurable functions.
 Let $(X,\La)$ be a right continuous, strong Markov process with
left-hand limits on $\R^d \times \ss$. The first component $X$
satisfies the following stochastic differential-integral equation
\begin{equation}
\label{eq:X}
\d X(t) = b(X(t), \La (t) ) \d t + \sigma(X(t), \La (t))\d W(t) + \int_{U} c(X(t-), \La (t-), u)\wdt N(\d t, \d u),
\end{equation}
where  $W$ is a standard $d$-dimensional Brownian motion,  $N$ is a Poisson random measure on $[0,\infty)\times U$ with intensity $\d t \,\nu(\d u)$, and  $\wdt N$ is the associated compensated Poisson random measure.
The second component $\Lambda$ is a
continuous-time random
process
taking values in the countably infinite
set $\ss$ such that
\begin{equation}\label{eq:La}
\P \{\La(t+\Delta)=l | \La(t)=k, X(t)=x \} =\begin{cases}
 q_{kl}(x) \Delta +o(\Delta),
&  \, \, \hbox{if}\, \, k \ne l, \\
1+q_{kk}(x) \Delta +o(\Delta), &  \, \, \hbox{if}\, \, k = l,
\end{cases}
\end{equation}
uniformly in $\R^d$, provided $\Delta \downarrow 0$.

To proceed, we construct a family of disjoint intervals $\{\Delta_{ij}(x): i,j \in \ss\}$ on the positive half real line as follows
\begin{align*}
  \Delta_{12}(x)  & = [0, q_{12}(x)),   \\
   \Delta_{13}(x)  & = [q_{12}(x)), q_{12}(x) + q_{13}(x)), \\
    &\ \   \vdots  \\
    \Delta_{21}(x) & = [q_{1}(x), q_{1}(x) + q_{21}(x)), \\
    \Delta_{23}(x) & = [q_{1}(x) + q_{21}(x)), q_{1}(x) + q_{21}(x) + q_{23}(x)), \\
    & \ \   \vdots   \\
      \Delta_{31}(x) & = [q_{1}(x) + q_{2}(x),q_{1}(x) +  q_{2}(x) + q_{31}(x)), \\
      & \ \   \vdots
\end{align*}
 where for convenience,
   we set
 $\Delta_{ij}(x) = \emptyset$ if $q_{ij}(x) = 0$, $i \not= j$.
 Note that for each $x\in \R^{d}$, $\{\Delta_{ij}(x): i,j \in \ss\}$
 are disjoint intervals, and the length of the interval $\Delta_{ij}(x)$
 is equal to $q_{ij}(x)$.
We then define a function $h$: $\R^{d} \times \ss \times \R_{+} \to \R$ by
\begin{equation}\label{eq-h-fn}h(x,k,r)=\sum_{l \in \ss}(l-k){\mathbf{1}}_{\Delta_{kl}(x)}(r).\end{equation} That is,
 for each $x\in \R^{d}$ and $ k \in \ss$,  we set $h(x,k,r)=l-k$ if $r \in\Delta_{kl}(x)$ for some $l\neq k$;
otherwise $h(x,k,r)=0$.   Consequently,
we can describe the evolution of  $\La$ using
the following stochastic differential equation \begin{equation}\label{eq:La-SDE}
 \La(t )=\La(0) + \int_{0}^{t }\int_{\R_{+}}h(X(s-),\La(s-),r){N}_{1}(\d s,\d r),
\end{equation}
 where $N_{1}$ is a Poisson random measure on $[0,\infty) \times [0,\infty)$ with characteristic measure $\m(\d z)$, the Lebesgue measure.

For convenience in the subsequent discussion,
let us give the infinitesimal generator $\A$
of the regime-switching jump diffusion $(X,\La)$
 \begin{equation}
\label{eq-operator}
\A f(x,k) : = \LL_{k}f(x,k) +  Q(x)f(x,k),
\end{equation}  with $a(x,k): = \sigma\sigma^{T}(x,k)$ and
\begin{align}
\label{eq-Lk-operator-defn}&
\begin{aligned} {\LL}_{k} f(x,k)
: =&\,  \frac {1}{2}\tr\bigl(a(x,k)\nabla^2 f(x,k)\bigr)+\langle b(x,k),
\nabla f(x,k)\rangle\\
  &  \, + \int_{U}
\bigl(f(x+c(x,k,u),k)-f(x,k)- \langle \nabla f(x,k), c(x,k,u)\rangle
\bigr)\nu(\d u), \end{aligned} \\
\label{eq-Qx-operator}
&\begin{aligned}  Q(x)f(x,k)  &:= \sum_{j\in \ss} q_{kj}(x) [f(x,j) - f(x,k)] \\
  & = \int_{[0,\infty)}[ f(x, k+ h(x,k,z)) - f(x,k)] \m(\d z).\end{aligned}
\end{align}
Define a metric $\lambda (\cdot,\cdot)$ on $\R^{d}
\times \ss$ as $\lambda \bigl((x,m), (y,n) \bigr)=|x-y|+d(m,n)$,
where $d(m,n) =\one_{\{ m\neq n\}} $ is the discrete
metric
on $\ss$. Let ${\B}(\R^{d} \times \ss)$ be the Borel $\sigma$-algebra on $\R^{d}
\times \ss$. Then $(\R^{d} \times \ss, \lambda (\cdot,\cdot), {\B}(\R^{d} \times \ss))$ is a locally compact and separable metric
space. For the existence and uniqueness of the strong Markov process
$(X,\La)$ satisfying
system \eqref{eq:X} and
\eqref{eq:La-SDE}, we make the following assumptions.

\begin{Assumption}\label{assumption-non-linear-growth}
{\rm There exists a  nondecreasing  function $\zeta:[0,\infty) \mapsto [1,\infty)$
 that is continuously differentiable and that satisfies
\begin{equation}\label{eq-zeta-sec-2}
 \int_{0}^{\infty} \frac{\d r}{ r\zeta(r) + 1} =\infty,
\end{equation}
such that for all 
$x\in \R^{d}$ and $k \in \ss$,
\begin{align}
\label{eq-coeffs-non-linear-growth}
& 2 \lan x, b(x,k)\ran + |\sigma(x,k)|^{2} + \int_{U}  |c(x,k,u)|^{2} \nu(\d u)  \le H  [|x|^{2}\zeta(|x|^{2}) + 1], \\
\label{q_k<Hk}
&q_{k}(x) :=-q_{kk}(x) = \sum_{l\in\ss\setminus \{k\}}q_{kl}(x)\le H k,  \\ \label{eq:switching-2nd-moment-condition}
& \sum_{l\in\ss\setminus \{k\}}(f(l)-f(k)) q_{kl} (x) \le H( 1+ \Phi(x) + f(k) ),
\end{align} where  $H $ is a positive  constant,
\begin{equation}
\label{eq-fn-Phi}
 \Phi(x): = \exp\bigg\{\int_{0}^{|x|^{2}}\frac{\d r}{r\zeta(r) +1}\bigg \}, \quad x\in \R^{d},
\end{equation}
and the function $f:\ss\mapsto \R_{+}$ is nondecreasing satisfying
$f(m)\to \infty$ as $m \to \infty$.  In addition, assume
 there exists some   $\delta \in (0, 1]$ such that \begin{equation}
\label{eq-q(x)-holder}
\sum_{l\in \ss\setminus\{k\}}| q_{kl}(x) - q_{kl}(y)| \le H |x-y|^{\delta}
\end{equation} for all $k\in \ss$ and $x,y\in \R^{d}$.
}\end{Assumption}

\begin{Assumption}\label{assumption-non-lip} {\rm Assume the following conditions hold.
\begin{itemize}
  \item  If $d =1$,
   then   there exist a positive number $\delta_{0}$  and  a nondecreasing   and  concave function $\varrho: [0,\infty)\mapsto [0,\infty)$ satisfying \begin{equation}
\label{eq-varrho-non-integrable}
 \int_{0^{+}} \frac{\d r}{\rho(r)} = \infty
\end{equation} such that  for all $k \in \ss$,   $R > 0$, and  $x,z\in \R$ with  $|x| \vee |z| \le R$ and $|x-z| \le \delta_{0}$, \begin{align}
\label{eq1-1d path-cdn}
&\sgn (x-z) (b(x,k)-b(z,k) )  \le  \kappa_{R} \varrho(|x-z|),\\
\label{eq2-1d path-cdn}  & |\sg(x,k)- \sg(z,k)|^{2} + \int_{U} |c(x,k,u)-c(z,k,u)|^{2}\nu(\d u) \le \kappa_{R} |x-z|,
\end{align}  where $\kappa_{R}$ is a positive constant and $\sgn(a)= 1$ if $a> 0$ and $-1$ if $a\le 0$. In addition, for each $k\in\ss$, the function $c$ satisfies that
\begin{equation}
\label{eq1-1d-c-fn-path-cdn}
  \text{the function } x\mapsto x+ c(x,k,u) \text{ is nondecreasing for all }u\in U; \end{equation}  or, there exists some    $ \beta >0 $ such that \begin{equation}
\label{eq2-1d-c-fn-path-cdn}
 |x-z+ \theta(c(x,k,u)- c(z,k,u))| \ge \beta |x-z|, \ \ \forall (x,z,u,\theta)\in \R\times\R\times U \times [0,1].
\end{equation}

  \item If $d \ge 2$, then there exist a positive number $\delta_{0}$,
  and a nondecreasing and concave function $\varrho: [0,\infty)\mapsto [0,\infty)$ satisfying 
   \begin{equation}\label{eq-varrho-conditions-Feller}0 < \varrho(r) \le  (1+r)^{2}\varrho(r/(1+r)) \text{ for all } r >  0,\,  \text{and}\, \int_{0^{+}} \frac{\d r}{\varrho(r)} = \infty
 \end{equation}
such that for all $k \in \ss$,   $R > 0$, and  $x,z\in \R^{d}$ with  $|x| \vee |z| \le R$ and $|x-z| \le \delta_{0}$, \begin{equation}
\label{eq-path-condition} \begin{aligned}
2   \lan x-z, & \ b(x,k)-b(z,k)\ran +  |\sigma(x,k) - \sigma(z,k)|^{2} \\&\quad
 + \int_{U}| c(x,k,u)-c(z,k,u) |^{2}  \nu(\d u) \le  \kappa_{R} \varrho(|x-z|^{2}),\end{aligned}
\end{equation}  where $\kappa_{R}$ is a positive constant.

\end{itemize}}\end{Assumption}

\begin{rem}{\rm
We make some comments  concerning Assumptions    \ref{assumption-non-linear-growth} and \ref{assumption-non-lip}.
 Examples of functions satisfying \eqref{eq-zeta-sec-2} include $\zeta(r) =1$,   $\zeta(r) = \log r$, and $  \zeta(r) = \log r \log(\log r) $ for $r$ large. When $\zeta(r) =1$, \eqref{eq-coeffs-non-linear-growth} reduces to the usual linear growth condition.  With other choices of $\zeta$, \eqref{eq-zeta-sec-2} allows super-linear condition for the coefficients of \eqref{eq:X} with respect to the variable $x$ 
 for each $k\in \ss$. This is motivated by applications such as Lotka-Volterra models, in which the coefficients have superlinear growth conditions.   Conditions \eqref{q_k<Hk} and \eqref{eq:switching-2nd-moment-condition} are imposed so that the $\La$ component will not explode in finite time with probability 1; see the proof of Theorem \ref{thm-soln-existence-uniqueness} for details.

 Examples of functions  satisfying \eqref{eq-varrho-non-integrable} or \eqref{eq-varrho-conditions-Feller}
 include $\varrho(r) = r$ and concave and increasing  functions such as $\varrho(r) = r \log(1/r)$,  $\varrho(r) = r \log( \log(1/r))$, and $\varrho(r) = r \log(1/r)  \log( \log(1/r))$ for $r \in (0,\delta)$ with $\delta > 0$ small enough. When $\varrho(r)= r$, Assumption  \ref{assumption-non-lip} is just the usual local Lipschitz condition. With other choices of continuity modularity, Assumption  \ref{assumption-non-lip} allows the drift, diffusion, and jump coefficients of \eqref{eq:X} to be non-Lipschitz with respect to the 
 variable $x$.  This, in turn, presents more opportunities
 for building
 realistic and flexible mathematical models for a wide range of applications. Indeed, non-Lipschitz coefficients are present in areas  such as
 branching diffusion in biology, the Cox-Ingersoll-Ross model in math finance, etc.

It is also worth pointing out that  \eqref{eq1-1d path-cdn}, \eqref{eq2-1d path-cdn}, and \eqref{eq-path-condition} of Assumption   \ref{assumption-non-lip}  only require the modulus of continuity to hold in a small neighborhood of the diagonal line $x=z$ in $\R^{d}\otimes\R^{d}$ with $|x|\vee |z| \le R$ for each $R > 0$. This is in contrast to those in \cite{LiPu-12} and adds some subtlety in the proof of pathwise uniqueness for \eqref{eq:X-k sde}.

When $d =1$, Assumption \ref{assumption-non-lip} allows the diffusion coefficient $\sigma(\cdot, k)$ to be locally H\"older continuous with exponent $\alpha \in [\frac 12, 1]$. This is the celebrated result in \cite{YamaW-71}. Such a result was extended to stochastic differential equations with jumps; see, for example, \cite{FuLi-10, LiMyt-11} and \cite{LiPu-12}, among others. In particular, \cite{LiPu-12} shows that if \eqref{eq1-1d-c-fn-path-cdn} holds, the function $x\mapsto \int_{U} c(x,k,u)\nu(\d u) $ can be locally  H\"older continuous with exponent $\alpha \in [\frac 12, 1]$ as well.   The continuity assumption \eqref{eq1-1d path-cdn} on the drift coefficient $b(\cdot, k)$ is slightly more general than that in \cite{LiPu-12}. In particular,   \eqref{eq1-1d path-cdn} will be satisfied as long as $b(\cdot, k)$ is decreasing.

}
\end{rem}

 \begin{lem}\label{lem-pathwise-unique}
Suppose  Assumption \ref{assumption-non-lip} and \eqref{eq-coeffs-non-linear-growth} hold. Then for each $k \in \ss$,   the stochastic differential equation
  \begin{equation}
\label{eq:X-k sde}  \begin{aligned}X(t) = x+ \int_{0}^{t}b(X(s),  k ) \d s +  \int_{0}^{t} \sigma(X(s), k)\d W(s) +  \int_{0}^{t}\int_{U} c(X(s-), k, u)\wdt N(\d s, \d u)
\end{aligned}\end{equation} has a unique non-explosive strong solution.
\end{lem}

 \begin{proof}
Condition \eqref{eq-coeffs-non-linear-growth} guarantees that the solution to \eqref{eq:X-k sde} will not explode in finite time with probability 1; see, for example, Theorem 2.1 in \cite{XiZ-18b}. When $d\ge2$, the existence and uniqueness of a strong solution to \eqref{eq:X-k sde} under Assumption \ref{assumption-non-lip}   follows from  Theorem 2.8 of \cite{XiZ-18b}.

 When $d =1$, we
 follow the arguments in the proof of Theorem 3.2 of \cite{LiPu-12} to show that  pathwise uniqueness holds for \eqref{eq:X-k sde}. First, let $\{a_{n}\}$ be a strictly decreasing sequence of real numbers satisfying $a_{0}=1$,  $\lim_{n\to\infty}a_{n} =0$, and $\int_{a_{n}}^{a_{n-1}} \frac{\d r}{r} = n$ for each $n \ge 1$.
For each $n \ge 1$, let $\rho_{n}$ be a nonnegative continuous function with support on $(a_{n}, a_{n-1})$ so that
\begin{displaymath}
\int_{a_{n}}^{a_{n-1}}\rho_{n}(r) \d r =1 \text{ and } \rho_{n}(r) \le 2(kr)^{-1} \text{ for all }r > 0.
\end{displaymath}
For $x\in \R$, define \begin{equation}
\label{eq-fn psi-n}
\psi_{n}(x) = \int_{0}^{|x|} \int_{0}^{y}\rho_{n}(z) \d z\d y.
\end{equation} We can immediately verify that $\psi_{n}$ is even and twice continuously differentiable, with
\begin{equation}\label{eq sgn of psi'}
\psi_{n}'(r) =
\sgn(r) \int_{0}^{|r|} \rho_{n}(z) \d z =\sgn(r) |\psi_{n}'(r)|,
\end{equation} and \begin{equation}
\label{eq-psi estimates}
 |\psi_{n}'(r)| \le 1,\quad  0 \le |r| \psi_{n}''(r) = |r| \rho_{n}(|r|) \le \frac2n,\quad  \text{and}\quad\lim_{n\to\infty} \psi_{n}(r) = |r|
\end{equation} for $r\in \R$. Furthermore, for each $r > 0$, the sequence $\{\psi_{n}(r) \}_{n\ge 1}$ is nondecreasing. Note also that for each $n\in \N$,  $\psi_{n}$, $\psi_{n}'$, and $\psi_{n}''$ all vanish  on the interval $(-a_{n}, a_{n})$.
Moreover the classical arguments reveal that
\begin{align}\label{eq-nonlocal term estimate}
\frac12&  \psi_{n}'' (x-z) |\sigma(x,k)-\sigma(z,k)|^{2} + \nonumber\int_{U}  [\psi_{n}(x-z + c(x,k,u)- c(z,k,u))\\ \nonumber& \qquad \qquad \qquad - \psi_{n}(x-z)- \psi_{n}' (x-z) ( c(x,k,u)- c(z,k,u))]\nu(\d u)
\\ & \le \frac12 \cdot \frac{2}{n}  \kappa_{R} +\frac{\kappa_{R}}{n} \bigg(\frac1\beta \vee 2\bigg) \le K \frac{\kappa_{R}}{n}, \end{align}
   for all $x,z$   with  $|x| \vee |z| \le R$
 and  $0<  |x-z|  \le \delta_{0}$, where $K $ is a positive constant independent of $R$ and $n$. On the other hand,   for any $x,z \in \R$ with $|x| \vee |z| \le R$ and $|x-z| \le \delta_{0}$, it follows    from \eqref{eq1-1d path-cdn} and \eqref{eq sgn of psi'} that  \begin{equation}
\label{eq-b difference}\begin{aligned}
\psi_{n}'(x-z) (b(x,k)- b(z,k)) & = \sgn(x-z) |\psi_{n}'(x-z)| (b(x,k)- b(z,k)) \\ & \le \kappa_{R} \varrho(|x-z|).
\end{aligned}\end{equation}

 Let $\wdt X$ and $X$ be two solutions to \eqref{eq:X-k sde}.  Denote $\Delta _{t} : = \wdt X(t) - X(t)$ for $t \ge 0$. Assume $|\Delta_{0}| = |\wdt x-x| < \delta_{0} $ and define $$S_{\delta_{0}}: = \inf\{t \ge 0: |\Delta_{t}| \ge \delta_{0}\}= \inf\{t \ge 0: |\wdt X(t)- X(t)| \ge \delta_{0}\}.$$
 For $R >0$, let $\tau_{R}: =\inf\{t \ge 0: |\wdt X(t)| \vee |X(t)| > R \}$. Then $\tau_{R} \to \infty$ a.s. as $R\to \infty$. Moreover, by It\^o's formula, we have
 \begin{align*}
\E [  \psi_{n}(\Delta_{t\wedge S_{\delta_{0}} \wedge\tau_{R}})]  & =  \psi_{n}(|\Delta _{0} |) + \E\biggl[\int_{0}^{t\wedge \tau_{R} \wedge S_{\delta_{0}}}\bigg\{ \psi_{n}' (\Delta_{s})[b(\wdt X(s),k)- b(X(s), k)]      \\
    &   \qquad + \frac12\psi_{n}'' (\Delta_{s})[\sigma(\wdt X(s),k)- \sigma(X(s), k)]^{2}\\
    & \qquad  + \int_{U} [\psi_{n}(\Delta_{s} + c(\wdt X(s), k, u)- c(\wdt X(s), k, u) )-\psi_{n}(\Delta_{s})\\
    & \qquad\qquad  -\psi_{n}'(\Delta_{s})(c(\wdt X(s), k, u)- c(\wdt X(s), k, u) ))  ]\nu(\d u)\bigg\} \d s\bigg]. \end{align*}
    Furthermore, using   \eqref{eq-nonlocal term estimate} and \eqref{eq-b difference}, we obtain
     \begin{align*}
\E [  \psi_{n}(\Delta_{t\wedge S_{\delta_{0}} \wedge\tau_{R}})]
    & \le  \psi_{n}(|\Delta _{0} |) + \E\biggl[\int_{0}^{t\wedge \tau_{R} \wedge S_{\delta_{0}}} \bigg( \kappa_{R} \varrho (|\Delta_{s}|) + K \frac{\kappa_{R}}{n}\bigg) \d s\bigg]\\
    & \le  \psi_{n}(|\Delta _{0} |) +K \frac{\kappa_{R}}{n}t + \int_{0}^{t} \kappa_{R} \rho\big(\E[|\Delta_{s\wedge \tau_{R} \wedge S_{\delta_{0}} }|]\big)\d s,
\end{align*} where the second inequality follows from the concavity of $\varrho$ and Jensen's inequality. Upon passing to the limit as $n\to \infty$, we obtain from the third equation in \eqref{eq-psi estimates} and the monotone convergence theorem that
\begin{displaymath}
\E[|\Delta_{t\wedge S_{\delta_{0}} \wedge\tau_{R}}|] \le |\Delta_{0}| + \kappa_{R}\int_{0}^{t}  \rho(\E[|\Delta_{s\wedge \tau_{R} \wedge S_{\delta_{0}} }|])\d s.
\end{displaymath} When $\Delta_{0} =0$, Bihari's inequality then implies that $\E[|\Delta_{t\wedge\tau_{R}\wedge S_{\delta_{0}}}|]=0$. Hence by Fatou's lemma, we have $\E[|\Delta_{t\wedge S_{\delta_{0}}}|]=0$. This implies that  $\Delta_{t\wedge S_{\delta_{0}} } =0$ a.s.

On the set $\{S_{\delta_{0}} \le t\}$, we have  $|\Delta_{t\wedge S_{\delta_{0}} } | \ge \delta_{0}$. Thus it follows that $0 =\E [ |\Delta_{t\wedge S_{\delta_{0}} } |] \ge \delta_{0} \P\{ S_{\delta_{0}} \le t\}$. Then, we have $ \P\{ S_{\delta_{0}} \le t\} =0$ and hence $\Delta_{t} =0$ a.s.
 The desired pathwise uniqueness
 for \eqref{eq:X-k sde} then follows from the fact that $\wdt X$ and $X$ have right continuous sample paths. Next
 similar to the proof of Theorem 5.1 of \cite{LiPu-12},
 \eqref{eq:X-k sde} has a weak solution,
  which further yields that
  the existence and uniqueness of a non-explosive strong solution to \eqref{eq:X-k sde}.
  \qed
\end{proof}

\begin{thm}\label{thm-soln-existence-uniqueness}
Under Assumptions  \ref{assumption-non-linear-growth} and \ref{assumption-non-lip}, for any $(x,k)\in \R^{d}\times \ss$, the system given by \eqref{eq:X} and \eqref{eq:La-SDE} has a unique non-explosive strong solution $(X,\La)$ with initial condition $(X(0),\La(0)) = (x,k)$.
\end{thm}

\begin{proof} The proof is divided into two steps.
First,
we
show that \eqref{eq:X} and \eqref{eq:La-SDE} has a non-explosive solution. The second step then
derives the
pathwise uniqueness
for \eqref{eq:X} and \eqref{eq:La-SDE}.  While the proof of the existence of a solution to \eqref{eq:X} and \eqref{eq:La-SDE}
use the same line of arguments as
in
the proof of Theorem 2.1 of \cite{XiZ-17}, some care are required here since
the assumptions in \cite{XiZ-17} have been relaxed.
 Moreover, an error
in the proof of \cite{XiZ-17} is corrected here.
The proof for pathwise uniqueness 
is more delicate than that in \cite{XiZ-17} since the global Lipschitz conditions with respect to the variable $x$
in \cite{XiZ-17} are no longer true in this paper.

 {\em Step 1.}    Let
 $(\Omega, {\F}, \{{\F}_{t}\}_{t\ge 0}, \, \P )$ be a   complete filtered probability space, on which are defined
 a $d$-dimensional standard Brownian motion $B$,
 and a  Poisson random measure $N(\cdot, \cdot)$ on $[0, \infty) \times U$
  with  a $\sigma$-finite characteristic measure $\nu$ on $U$.
 In addition, let $\{\xi_{n}\}$ be a sequence of independent
 exponential random variables with mean $1$
 on $(\Omega, {\F}, \{{\F}_{t}\}_{t\ge 0}, \, \P )$
 that is independent of $B$ and $N$.

Let $k \in \ss$ and consider
 the stochastic differential equation
  \begin{equation}
\label{eq:X-k}  \begin{aligned}X^{(k)}(t) = x+ \int_{0}^{t}b(X^{(k)}(s),  k ) \d s +  \int_{0}^{t} \sigma(X^{(k)}(s), k)\d W(s) \\ +  \int_{0}^{t}\int_{U} c(X^{(k)}(s-), k, u)\wdt N(\d s, \d u).
\end{aligned}\end{equation}  
Lemma \ref{lem-pathwise-unique}  guarantees that
SDE \eqref{eq:X-k} has a unique non-explosive strong solution $X^{(k)}$.   As in the proof of Theorem 2.1 of \cite{XiZ-17}, we define   \begin{equation}
\label{eq-tau1-defn}
 \tau_{1}= \theta_{1}: = \inf\set{t\ge 0: \int_{0}^{t} q_{k}(X^{(k)}(s))\d s > \xi_{1}}.
\end{equation} 
Thanks to \eqref{q_k<Hk}, we have   $\P(\tau_{1} > 0) =1$.
We define a process $(X,\La) $ on $[0, \tau_{1}]$ as
\begin{displaymath}
 X(t) = X^{(k)}(t) \text{ for all } t \in [0, \tau_{1}], \text{ and }\La(t)  =    k   \text{ for all } t \in [0, \tau_{1}).
\end{displaymath}
   Moreover, we define $\La(\tau_{1})\in \ss$ according to the probability distribution \begin{equation}\label{eq-switching-to-location}
 \P\set{\La(\tau_{1}) = l| \F_{\tau_{1}-}} = \dfrac{q_{kl}(X(\tau_{1}-))}{q_{k}(X(\tau_{1}-))} (1- \delta_{kl}) \one_{\{q_{k}(X(\tau_{1}-)) > 0  \}} + \delta_{kl} \one_{\{q_{k}(X(\tau_{1}-)) = 0  \}},
 \end{equation} for $l \in \ss.$  In general, having determined $(X,\La)$ on $[0, \tau_{n}]$, we let
 \begin{equation}
 \label{eq-theta-n+1-defn}
\theta_{n+1}: = \inf\biggl\{t\ge 0: \int_{0}^{t} q_{\La(\tau_{n})}(X^{(\La(\tau_{n}))}(s))\d s > \xi_{n+1}\biggr\}, \end{equation}   where
 \begin{displaymath}
\begin{aligned}
   X^{(\La(\tau_{n}))}(t)  & : =   X(\tau_{n}) +
\displaystyle \int_{0}^{t} \sigma (X^{(\La(\tau_{n}))}(s),\La(\tau_{n}))\d B(s)+ \int_{0}^{t} b(X^{(\La(\tau_{n}))}(s),\La(\tau_{n}))\d s
  \\ & \qquad+\displaystyle \int_{0}^{t} \int_{U} c(X^{(\La(\tau_{n}))}(s-),\La(\tau_{n}),u)
\wdt{N}(\d s,\d u).\end{aligned}
\end{displaymath} As before,  \eqref{q_k<Hk}    implies that $\P\{\theta_{n+1} > 0 \} =1$.
Then we let \begin{equation}
\label{eq-tau-n+1-defn}
 \tau_{n+1} : = \tau_{n} + \theta_{n+1}
\end{equation} and define $(X,\La)$ on $[\tau_{n}, \tau_{n+1}]$ by
\begin{align}\label{eq-XLa-nth-segment}
X(t) =  X^{(\La(\tau_{n}))}(t -\tau_{n})  \text{ for } t\in  [\tau_{n}, \tau_{n+1}], \, \, \La(t) = \La(\tau_{n})  \text{ for } t\in  [\tau_{n}, \tau_{n+1}), \
\end{align}
and
\begin{equation}\label{eq-sw-n-mechanism}\begin{aligned}  \P \set{\La(\tau_{n+1}) = l| \F_{\tau_{n+1}-}}    & = \delta_{\La(\tau_n),l} \one_{\{q_{\La(\tau_n)}(X(\tau_{n+1}-)) = 0  \}}\\ & +  \dfrac{q_{\La(\tau_n),l}(X(\tau_{n+1}-))}{q_{\La(\tau_n)}(X(\tau_{n+1}-))} (1- \delta_{\La(\tau_n),l})  \one_{\{q_{\La(\tau_n)}(X(\tau_{n+1}-)) > 0  \}} .
\end{aligned} \end{equation}
As argued in \cite{XiZ-17}, this ``interlacing procedure''
  uniquely determines a solution $(X,\La)\in \R^{d}\times \ss$ to  \eqref{eq:X} and \eqref{eq:La-SDE} for all $t \in[0,\tau_{\infty})$, where
 \begin{equation}
\label{eq-tau-infty-defn}
 \tau_{\infty}=\lim_{n\to \infty}\tau_n.
\end{equation} Since the sequence $\tau_{n}$ is strictly increasing, the limit $\tau_{\infty} \le \infty$ exists.

Next we
show that $\tau_{\infty} = \infty$ a.s.  To this end, fix $(X(0),\La(0)) = (x,k)\in \R^{d}\times \ss$ as in Step 1 and for any $m \ge k+1$, denote by $\wdt{\tau}_m :=\inf \{t\ge 0: \La(t)\ge m\}$ the first exit time for the $\La$ component from the finite set $\{0, 1, \dots, m-1\}$.  Let $A^{c}: =\{ \omega\in \Omega: \tau_{\infty} > \wdt \tau_{m} \text{  for all }m \ge k+ 1\}$ and $A : = \{ \omega\in \Omega: \tau_{\infty} \le \wdt \tau_{m_{0}} \text{ for some }m_{0} \ge k+1\}$. Then we have
  \begin{equation}
\label{eq-tau-infty=infty}
\P\{\tau_{\infty} = \infty\} = \P\{ \tau_{\infty} = \infty |A^{c}\} \P(A^{c}) + \P\{ \tau_{\infty} = \infty| A\} \P(A) .
\end{equation} 

Let $A_{m}:=\big\{\omega\in \Omega: \tau_{\infty} \le \wdt{\tau}_{m}\big\}$ for $m\ge k+1$. Then $A=\bigcup_{m=k+1}^{\infty}A_{m},  $ and $A^{c}=\bigcap_{m=k+1}^{\infty}A_{m}^{c}.$ Also denote $B_{k+1}:=A_{k+1}$ and let $$B_{m}:=A_{m}\setminus A_{m-1}=\big\{\omega\in \Omega: \wdt{\tau}_{m-1} <\tau_{\infty} \le \wdt{\tau}_{m}\big\}$$ for $m\ge k+2$. Clearly, $\{B_{m}\}_{m=k+1}^{\infty}$ is a sequence of disjoint sets and we have
\begin{equation}\label{Adisjointsum} A:=\bigcup_{m=k+1}^{\infty}B_{m}.\end{equation}

We proceed to
show that $\P\{ \tau_{\infty} = \infty| B_{m}\} =1$ for each $m$. 
   Note that  on the set $B_{m}$, $\La(\tau_{n}) \le m$ for all $n =1, 2,\dots$
Consequently, using
\eqref{q_k<Hk} in Assumption \ref{assumption-non-linear-growth}, we have $q_{\La(\tau_{n})}(X^{(\La(\tau_{n}))}(s)) \le Hm$ for all $n$ and $s\ge 0$.  On the other hand, thanks to the definition of $\theta_{1}$ in \eqref{eq-tau1-defn}, for any $\e > 0$, we have \begin{displaymath}
\xi_{1} < \int_{0}^{\theta_{1}+ \e} q_{k}(X^{(k)}(s)) \d s.  
\end{displaymath} Consequently, it follows that \begin{displaymath}
\one_{B_{m}}\xi_{1} \le\one_{B_{m}} \int_{0}^{\theta_{1}+ \e} q_{k}(X^{(k)}(s)) \d s \le \one _{B_{m}}  Hm (\theta_{1}+ \e).
\end{displaymath}
In the same manner, we have from \eqref{eq-theta-n+1-defn} that  \begin{displaymath}
\xi_{n} < \int_{0}^{\theta_{n}+ \e/2^{n}} q_{\La(\tau_{n-1})}(X^{(\La(\tau_{n-1}))}(s))\d s, \end{displaymath} and hence \begin{displaymath}
\one_{B_{m}} \xi_{n}  \le \one_{B_{m}}\int_{0}^{\theta_{n}+ \e/2^{n}} q_{\La(\tau_{n-1})}(X^{(\La(\tau_{n-1}))}(s))\d s \le  \one_{B_{m}}   Hm (\theta_{n}+ \e/2^{n}), \ \ \forall n =1, 2,\dots
\end{displaymath}
Summing over these inequalities and noting that $\tau_{\infty}= \sum_{n=1}^{\infty} \theta_{n}$, we arrive at
\begin{equation}\label{ineq-tau-infty-observation}
\one_{B_{m}} \sum_{n=1}^{\infty} \xi_{n} \le \one_{B_{m}} Hm (\tau_{\infty} + 2 \e).
\end{equation} By virtue of
Theorem 2.3.2 of \cite{Norris-98}, we have  $\sum_{n=1}^{\infty} \xi_{n}  = \infty$ a.s. Therefore it follows that
 $\P(\sum_{n=1}^{\infty} \xi_{n}  = \infty|B_{m}) =1$. Then \eqref{ineq-tau-infty-observation} implies that \begin{displaymath}
\P\{\tau_{\infty}  = \infty|B_{m} \} \ge \P\Bigg\{ \sum_{n=1}^{\infty} \xi_{n}  = \infty\big |B_{m}\Bigg\}  =1,
\end{displaymath} as desired. Consequently, we can use \eqref{Adisjointsum} to compute
\begin{align}\label{eq-t-tau-infty=infty|A}
\nonumber \P\{ \tau_{\infty} = \infty| A\}    &=  \frac{\P\{\tau_{\infty} = \infty, A \}}{ \P(A)} =  \frac{\P\{\tau_{\infty} = \infty, \bigcup_{m=k+1}^{\infty}B_{m} \}}{ \P(A)}  \\ \nonumber &= \frac{ \sum_{m=k+1}^{\infty}\P\{\tau_{\infty} = \infty, B_{m} \} }{\P(A)}
       =\frac{ \sum_{m=k+1}^{\infty}\P\{\tau_{\infty} = \infty| B_{m} \} \P(B_{m}) }{\P(A)} \\ & =    \frac{ \sum_{m=k+1}^{\infty}  \P(B_{m}) }{\P(A)}=1.
\end{align}

If $\P(A) =1$ or $\P(A^{c}) =0$, then \eqref{eq-tau-infty=infty} and \eqref{eq-t-tau-infty=infty|A} imply that $\P\{\tau_{\infty} = \infty \} =1$ and the proof is complete. Therefore, it remains to consider the case when $\P(A^{c}) >0$.  Denote $\wdt\tau_{\infty}: =\lim_{m\to \infty}\wdt \tau_{m}$.  Note that $A^{c}= \{ \tau_{\infty} \ge \wdt \tau_{\infty}\}$. Thus $\P\{\tau_{\infty} = \infty |A^{c} \}\ge \P\{\wdt \tau_{\infty} = \infty |A^{c} \}$ and hence \eqref{eq-tau-infty=infty}
holds if we can show  that \begin{equation}
\label{eq-tilde-tau-infty=infty}
\P\{\wdt \tau_{\infty} = \infty |A^{c} \} =1.
\end{equation} Assume on the contrary that \eqref{eq-tilde-tau-infty=infty}
were false, then there would exist a $T > 0$ such that
\begin{displaymath}
\delta: = \P\{\wdt \tau_{\infty} \le T, A^{c} \} > 0.
\end{displaymath}
Let $f:\ss\mapsto \R_{+}$ be as in Assumption \ref{assumption-non-linear-growth}.  Then we have for any $m \ge k+1$,
\begin{align*}
 f(k) & = \E[e^{-H(T\wedge \tau_{\infty}\wedge \wdt \tau_{m})} f(\La(T\wedge \tau_{\infty}\wedge \wdt \tau_{m}))]  \\
               & \quad + \E\biggl[ \int_{0}^{T\wedge \tau_{\infty}\wedge \wdt \tau_{m}} e^{-Hs}\biggl(H f(\La(s)) - \sum_{l\in \ss} q_{\La(s), l}(X(s)) [f(l) - f(\La(s))] \biggr) \d s\biggr] \\
 &  \ge \E[e^{-H(T\wedge \tau_{\infty}\wedge \wdt \tau_{m})} f(\La(T\wedge \tau_{\infty}\wedge \wdt \tau_{m}))]  \\
               & \quad +  \E\biggl[ \int_{0}^{T\wedge \tau_{\infty}\wedge \wdt \tau_{m}} e^{-Hs} [H f(\La(s)) - H (1+\Phi( X(s)) + f(\La(s))) ] \d s \biggr]  \\
               & \ge  \E[e^{-H(T\wedge \tau_{\infty}\wedge \wdt \tau_{m})} f(\La(T\wedge \tau_{\infty}\wedge \wdt \tau_{m}))],
\end{align*} where the  first inequality above follows from \eqref{eq:switching-2nd-moment-condition} in Assumption  \ref{assumption-non-linear-growth}.
Consequently, we have
\begin{equation}
\label{eq-exp-HT-f(k)}
\begin{aligned}
 e^{HT} f(k)& \ge   \E[ f(\La(T\wedge \tau_{\infty}\wedge \wdt \tau_{m}))] \ge \E[f(\La(\wdt\tau_{m})) \one_{\{\wdt\tau_{m} \le T\wedge \tau_{\infty}\}}] \\ & \ge f(m) \P\{\wdt\tau_{m} \le T\wedge \tau_{\infty}\}   \ge f(m) \P\{\wdt\tau_{m} \le T\wedge \tau_{\infty}, A^{c}\} \\ & \ge  f(m) \P\{\wdt\tau_{\infty} \le T\wedge \tau_{\infty}, A^{c}\},
\end{aligned}
\end{equation}
where the third inequality follows from the facts that $\La(\wdt \tau_{m}) \ge m$ and that $f $ is nondecreasing, and the last inequality follows from the fact that $\wdt \tau_{m} \uparrow \wdt \tau_{\infty} $.  Recall that $A^{c} = \{\tau_{\infty} \ge \wdt \tau_{\infty} \}$. Thus \begin{align*}
 \P\{\wdt\tau_{\infty} \le T\wedge \tau_{\infty}, A^{c}\} &= \P \{\wdt\tau_{\infty} \le T\wedge \tau_{\infty}, \wdt \tau_{\infty} \le   \tau_{\infty}  \} \\ &\ge \P \{\wdt\tau_{\infty} \le T, \wdt \tau_{\infty} \le \tau_{\infty} \} = \P\{\wdt\tau_{\infty} \le T, A^{c} \} = \delta> 0.
\end{align*}
Using this observation in \eqref{eq-exp-HT-f(k)} yields $\infty > e^{HT} f(k) \ge f(m) \delta \to \infty$ as $m\to \infty$, thanks to the fact that $f(m) \to \infty$ as $m\to \infty$, which
is a contradiction.
This establishes \eqref{eq-tilde-tau-infty=infty} and hence $\P(\tau_{\infty} = \infty) =1$. In other words, the interlacing procedure uniquely determines a solution $(X,\La) = (X^{(x,k)},\La^{(x,k)})$ for all $t\in [0,\infty)$.

Next we show that the solution $(X,\La)$ to the system \eqref{eq:X} and \eqref{eq:La-SDE} is
non-explosive a.s.
Consider the function $V(x,k) : = 1 + \Phi(x) + f(k)$, where the functions $\Phi: \R^{d}\mapsto \R_{+} $ of \eqref{eq-fn-Phi} and $f: \ss\mapsto \R_{+}$ are defined in Assumption \ref{assumption-non-linear-growth}. Note that  $V(x,k) \to \infty$ as $|x|\vee k \to \infty$ thanks to  Assumption \ref{assumption-non-linear-growth}. Using the definition of $\A$ of \eqref{eq-operator}, we have $$\A V(x,k) = \LL_{k} \Phi(x) + Q(x) f(k) .$$  Moreover, detailed computations using \eqref{eq-zeta-sec-2} and \eqref{eq-coeffs-non-linear-growth} reveal that $\LL_{k}\Phi(x) \le H  \Phi(x)$ for all $x\in \R^{d}$ and $k\in \ss$. On the other hand, \eqref{eq:switching-2nd-moment-condition} implies that  $ Q(x) f(k) \le H(1+ \Phi(x) + f(k))$. Combining these estimates, we obtain $\A V(x,k) \le 2 H V(x,k)$.  This, together with It\^o's formula,  shows that the process $\{e^{-2 H t} V(X(t), \La(t)), t \ge 0\}$ is a nonnegative local supermartingale.
  Then we can apply  the optional sampling theorem to the process $\{e^{-2 H t} V(X(t), \La(t)), t \ge 0\}$ to argue that $\P\{ \lim_{n\to\infty} T_{n} =\infty\} =1$, where $T_{n}: = \inf\{t \ge 0: |X(t)| \vee \La(t) \ge n \}$.  This shows that the solution $(X,\La)$  has no finite explosion time a.s.

{\em Step 2.} 
Suppose $(X,\La)$ and $(\wdt X, \wdt \La)$ are two solutions to
\eqref{eq:X} and \eqref{eq:La-SDE} starting from the same initial condition $(x,k) \in \R^{d}\times \ss$. Then we have
\begin{align*}
 \wdt X(t) - X(t) & = \int_{0}^{t} [b(\wdt X(s),\wdt\La(s)) - b(X(s),\La(s))]\d s  \\ & \quad
 + \int_{0}^{t}  [\sigma(\wdt X(s),\wdt\La(s)) - \sigma(X(s),\La(s))]\d W(s) \\ & \quad + \int_{U} [c(\wdt X(s-),\wdt \La(s-), z) - c(X(s-),\La(s-),z)] \wdt N(\d s, \d u), \end{align*}
   and
\begin{align*}  \wdt \La(t) - \La(t) =\int_{0}^{t} \int_{\R_{+}}[h(\wdt X(s-),\wdt \La(s-), z) - h(X(s-),\La(s-),z)] N_{1} (\d s, \d z).
\end{align*}

Let $\zeta:=\inf\{ t \ge 0: \La(t) \neq \wdt \La(t)\}$ be the first time when the discrete components differ from each other. Let us also  define $T_{R} : =\inf\{t \ge 0: |\wdt X(t)| \vee |X(t)| \vee \wdt \La(t) \vee \La(t) \ge R\}$ for $R > 0$ and   $S_{\delta_{0}}: = \inf\{t \ge 0: |\wdt X(t)- X(t)|  \ge \delta_{0} \}$. We have $\La(t) = \wdt \La(t)$ for $t \in [0, \zeta)$. To simplify notation, let us define $\Delta_{t}: = \wdt X(t)- X(t)$. Then from the proof of Theorem 2.6 of \cite{XiZ-18b}, we have  $\E [H( |\Delta_{t \wedge \zeta\wedge S_{\delta_{0}}} |)] =0$ for $d \ge 2$, where $H(r) : = \frac{r^{2}}{1+r^{2}}$, $r\ge 0$.  When $d=1$, the proof of Lemma \ref{lem-pathwise-unique}  reveals that $\E[|\Delta_{t \wedge \zeta\wedge S_{\delta_{0}}} |] =0$ and hence $\E [H( |\Delta_{t \wedge \zeta\wedge S_{\delta_{0}}} |)] =0$.  Note that on the set $\{ S_{\delta_{0} } \le t \wedge \zeta\}$, $ |\Delta_{t \wedge \zeta\wedge S_{\delta_{0}}} | \ge \delta_{0}$. Also we can readily check that $H$ is an increasing function. Thus, it follows that \begin{align*}
0    & = \E [H( |\Delta_{t \wedge \zeta\wedge S_{\delta_{0}}} |)]  \ge \E[H( |\Delta_{t \wedge \zeta\wedge S_{\delta_{0}}} |) \one_{\{ S_{\delta_{0} } \le t \wedge \zeta\}} ] \ge H(\delta_{0}) \P\{S_{\delta_{0} } \le t \wedge \zeta \}.   \end{align*}
This implies that $\P\{S_{\delta_{0} } \le t \wedge \zeta \}  =0$. Consequently, we have \begin{align*}
  \E[H(|\Delta_{t \wedge \zeta  }|)]  & =      \E[H(|\Delta_{t \wedge \zeta} |)  \one_{\{ S_{\delta_{0} } \le t \wedge \zeta\}}]  +    \E[H(|\Delta_{t \wedge \zeta}|)  \one_{\{ S_{\delta_{0} } > t \wedge \zeta\}}]    \\
    &   \le \P\{S_{\delta_{0} } \le t \wedge \zeta \}  + \E[H( |\Delta_{t \wedge \zeta\wedge S_{\delta_{0}}} |) \one_{\{ S_{\delta_{0} } > t \wedge \zeta\}}]\\
    & \le 0.
\end{align*}
It follows that $\E [ |\Delta_{t \wedge \zeta }|] =\E [ |\wdt X(t\wedge \zeta) - X(t\wedge \zeta)|  ] =0$.   Then we have  \begin{equation}
\label{eq-difference-1st-moment}
 \E [ |\wdt X(t\wedge \zeta) - X(t\wedge \zeta)|^{\delta} ]= 0,
\end{equation} where $\delta \in (0, 1]$ is the H\"{o}lder constant in \eqref{eq-q(x)-holder}.

Note that $\zeta \le t $ if and only if $\wdt \La(t\wedge \zeta) - \La(t\wedge \zeta)\neq 0$. Therefore, it follows that
\begin{align*}
\P&\{\zeta \le t\}    =\E[\one_{ \{\wdt \La(t\wedge \zeta) - \La(t\wedge \zeta)\neq 0\}}]   \\
 & = \E \biggl[\int_{0}^{t\wedge \zeta} \int_{\R_{+}} (\one_{\{\wdt \La(s-) - \La(s-) + h(\wdt X(s-), \La(s-), z) - h(X(s-),\La(s-), z) \neq 0 \}}  \\ & \qquad \qquad\qquad\qquad-  \one_{\{\wdt \La(s-) - \La(s-)\neq 0 \}})  \m(\d z) \d s\biggr] \\
 & = \E\biggl[\int_{0}^{t\wedge \zeta} \int_{\R_{+}}  \one_{\{h(\wdt X(s-), \La(s-), z) - h(X(s-),\La(s-), z) \neq 0 \}} \m(\d z) \d s\biggr] \\
 & \le  \E\biggl[\int_{0}^{t\wedge \zeta} \sum_{l \in \ss, l \neq \La(s-)} | q_{\La(s-),l}(\wdt X(s-)) - q_{\La(s-),l}(X(s-))| \d s\biggr]\\
 & \le H \E \biggl[\int_{0}^{t\wedge \zeta} | \wdt X(s-) ) -  X(s-) |^{\delta} \d s\biggr]  \le H \int_{0}^{t} \E[|\wdt X(s\wedge \zeta) - X(s\wedge \zeta) |^{\delta}] \d s =0,
\end{align*}  where the second inequality follows from \eqref{eq-q(x)-holder}.
In particular, we have
\begin{equation}
\label{eq-wdtLA=LA}
   \E[\one_{\{\wdt \La(t) \neq \La(t) \}}] \le \P\{\zeta \le t \}=0.
\end{equation}
Now we can compute
\begin{align*}
\E[H( |\wdt X(t) - X(t) |) ] & =\E[ H( |\wdt X(t) - X(t) | )\one_{\{\zeta > t \}}] + \E[H( |\wdt X(t) - X(t) |)\one_{\{\zeta \le t \}}]   \\
 & =\E[H( |\wdt X(t\wedge \zeta) - X(t\wedge \zeta) |) \one_{\{\zeta > t \}}] + \E[ 1 \cdot \one_{\{\zeta \le t \}}]\\
 & \le \E[ H(|\wdt X(t\wedge \zeta) - X(t\wedge \zeta) |) ] + 0 \\
 & =0.
\end{align*}
Thus $\P\{ \wdt X(t) = X(t)  \} =1$. This, together with \eqref{eq-wdtLA=LA}, implies that    $\P \{ (\wdt X(t), \wdt\La(t))=  ( X(t),\La(t))\} =1$ for all $t\ge 0$. Since the sample paths of $(X,\La)$ are right continuous,  we obtain the desired pathwise uniqueness result. \qed \end{proof}


\begin{exm}\label{exm1}{\rm
Let us consider the following SDE \begin{equation}
\label{eq-example1}\begin{aligned}
\d X(t) =&\ b(X(t),\La(t))\d t + \sigma(X(t), \La(t))\d W(t)   \\ & \quad  + \int_{U} c(X(t-),\La(t-), u) \wdt N(\d t, \d u), \ \  X(0) = x \in \R^{3},
\end{aligned}\end{equation} where $W$ is a 3-dimensional standard Brownian motion, $\wdt N(\d t,\d u)$ is a compensated Poisson random measure with compensator $\d t\,\nu(\d u)$ on $[0,\infty)\times U$, in which $U=  \{u\in \R^{3}: 0 < |u| <1 \}$ and $\nu(\d u) : = \frac{\d u}{|u|^{3+\alpha}}$ for some $\alpha\in (0,2)$. The $\La$ component in \eqref{eq-example1} takes value in $\ss= \{1,2,\dots\}$ and is generated by $Q(x) = (q_{kl}(x))$, with $q_{kl} (x) = \frac{k}{2^{l}}\cdot \frac{|x|^{2}}{1+|x|^{2}}$ for $x \in \R^{3}$ and $k \neq l \in \ss$. Let $q_{k}(x) = - q_{kk}(x) = \sum_{l\neq k} q_{kl}(x)$.
The coefficients of \eqref{eq-example1} are given by
\begin{displaymath}
b(x,k)=\begin{pmatrix}-x_{1}^{1/3}-k x_{1}^{3}\\ -x_{2}^{1/3}-k x_{2}^{3}\\-x_{3}^{1/3}-k x_{3}^{3}\\ \end{pmatrix}\!, \quad
c(x,k,u)=c(x,u)= \begin{pmatrix} \gamma x_{1}^{2/3} |u| \\  \gamma x_{2}^{2/3} |u|\\ \gamma x_{3}^{2/3} |u| \end{pmatrix}\!,
\end{displaymath} and 
\begin{displaymath}
\sigma(x,k) = \begin{pmatrix}\frac{x_{1}^{2/3}}{\sqrt 2} + 1 & \frac{\sqrt{ k}\, x_{2}^{2}}{3} & \frac{\sqrt{ k}\, x_{3}^{2}}{3}\\    \frac{\sqrt{ k}\, x_{1}^{2}}{3} &\frac{x_{2}^{2/3}}{\sqrt 2} + 1& \frac{\sqrt{ k}\, x_{3}^{2}}{3} \\ \frac{\sqrt{ k}\, x_{1}^{2}}{3} & \frac{\sqrt{ k}\, x_{2}^{2}}{3} & \frac{x_{3}^{2/3}}{\sqrt 2} + 1& \end{pmatrix}\!,
\end{displaymath} in which $\gamma$ is a positive constant so that $\gamma^{2} \int_{U} |u|^{2} \nu(\d u) = \frac12$.

Note that $\sigma$ and $b$ grow very fast in the neighborhood of $\infty$ and they
are H\"older continuous with orders $\frac23$
and $ \frac13$, respectively.
Nevertheless,   the coefficients of \eqref{eq-example1} still satisfy  Assumptions  \ref{assumption-non-lip}  and  \ref{assumption-non-linear-growth} and hence a unique non-exploding strong solution of \eqref{eq-example1} exists.  The verifications of these assumptions are as follows.
\begin{align*}
 \nonumber2& \lan x, b(x,k )\ran +|\sigma(x,k)|^{2} + \int_{U} |c(x,k,u)|^{2} \nu(\d u)        \\
  \nonumber  &    = 2 \sum_{j=1}^{3} x_{j}\bigl(-x_{j}^{1/3}-k x_{j}^{3}\bigr)  + \sum_{j=1}^{3} \biggl(\frac12 x_{j}^{4/3} + \frac{2k}{9} x_{j}^{4} +\sqrt 2 x_{j}^{2/3}+ 1\biggr)\\ \nonumber& \qquad + \int_{U} \gamma^{2} |u|^{2} \sum_{j=1}^{3} x_{j}^{4/3} \nu(\d u)\\
    & = - \frac{16k}{9}  \sum_{j=1}^{3} x_{j}^{4} - \sum_{j=1}^{3} x_{j}^{4/3} + \sqrt 2 \sum_{j=1}^{3} x_{j}^{2/3}  + 3. 
\end{align*}
 Thus \eqref{eq-coeffs-non-linear-growth} of Assumption \ref{assumption-non-linear-growth} hold. Furthermore, \eqref{q_k<Hk} is trivially satisfied. Consider the function $f(l) =  l$, $l\in \ss$. We have \begin{displaymath}
\sum_{l\neq k} (f(l) - f(k)) q_{kl}(x) = \sum_{l\neq k}( l-k) \frac{k}{2^{l}} \frac{|x|^{2}}{1+|x|^{2}}\le \sum_{l\neq k} l \frac{k}{2^{l}} \frac{|x|^{2}}{1+|x|^{2}} \le k \sum_{l \in \ss} \frac{l}{2^{l}} = 2 k,
\end{displaymath}
which yields
\eqref{eq:switching-2nd-moment-condition}. If $x, y \in \R^{3}$, we 
obtain
\begin{align*}
 \sum_{l\neq k} |q_{kl}(x) -q_{kl}(y) |   & = \sum_{ l\neq k} \frac{l}{2^{k}} \bigg| \frac{|x|^{2}}{1+ |y|^{2}} - \frac{|x|^{2}}{1+ |y|^{2}} \bigg|     = \sum_{ l\neq k} \frac{l}{2^{k}} \frac{||x| - |y|| (|x| + |y|)}{(1+ |x|^{2}) (1+ |y|^{2})}\\
    &   \le  \sum_{ l\neq k} \frac{l}{2^{k}} |x-y| \bigg(\frac{|x|}{1+ |x|^{2}} + \frac{|y|}{1+ |y|^{2}}\bigg)  \le 2 |x-y|.
\end{align*} This establishes \eqref{eq-q(x)-holder} and therefore verifies Assumption \ref{assumption-non-linear-growth}.

For the verification of Assumption  \ref{assumption-non-lip}, we compute
\begin{align*}
 \nonumber &2  \lan x-y, b(x,k)-b(y,k)\ran +|\sigma(x,k)-\sigma(y,k)|^{2} + \int_{U} |c(x,k,u)-c(y,k,u)|^{2} \nu(\d u)          \\
 \nonumber   & = -2  \sum_{j=1}^{3} (x_{j}- y_{j})(x_{j}^{1/3} - y_{j}^{1/3} +k x_{j}^{3} - k y_{j}^{3}) +\frac12  \sum_{j=1}^{3}(x_{j}^{2/3} - y_{j}^{2/3})^{2} \\  \nonumber & \qquad + \frac{2k}{9}\sum_{j=1}^{3}(x_{j}^{2}- y_{j}^{2})^{2} + \int_{U} \sum_{j=1}^{3} \gamma^{2} (x_{j}^{2/3} - y_{j}^{2/3})^{2} |u|^{2} \nu(\d u)\\
    & = - \frac{16k}{9}\sum_{j=1}^{3}  (x_{j}- y_{j})^{2}\biggl[\biggl(x_{j}+\frac{7}{16}y_{j}\biggr)^{2} + \frac{207}{256} y_{j}^{2} \biggr]- \sum_{j=1}^{3}\bigl(x_{j}^{1/3}- y_{j}^{1/3}\bigr)^{2} \bigl(x_{j}^{2/3}+y_{j}^{2/3}\bigr).
    \end{align*} Obviously  this implies  \eqref{eq-path-condition} and thus verifies Assumption   \ref{assumption-non-lip}.}
\end{exm}

\section{Feller Property}\label{sect-Feller}
In Section \ref{sect-str-formulation}, we established the  existence and uniqueness of a solution in the strong sense to
system \eqref{eq:X} and \eqref{eq:La-SDE} under Assumptions \ref{assumption-non-linear-growth} and \ref{assumption-non-lip}. The solution $(X,\La)$ is a two-component c\`adl\`ag strong Markov process.  In this section, we study the Feller property for such processes. For any $ f \in C_{b}(\R^{d}\times \ss)$, by the continuity of $f$ and the right continuity of the sample paths of $(X,\La)$, we can use   the bounded convergence theorem to obtain $\lim_{t\downarrow 0} \E_{x,k} [f(X(t),\La(t))] = f(x,k)$. Therefore the process $(X,\La)$ satisfies the Feller property if   the semigroup $P_{t} f(x,k): = \E_{x,k}[f(X(t),\La(t))], f \in \B_{b}(\R^{d}\times \ss)$ maps $C_{b}(\R^{d}\times \ss) $ into itself. Obviously, to establish the Feller property, we
only
need
the distributional properties of the process $(X,\La)$. Thus in lieu of the strong formulation used in Section  \ref{sect-str-formulation}, we will assume the following
 ``weak formulation'' throughout the section.


\begin{Assumption}\label{Assumption-wk soln-well-posed}{\rm
For any initial data
$(x,k)\in \R^{d}\times\ss$, the system of stochastic differential equations \eqref{eq:X} and \eqref{eq:La-SDE} has a non-exploding weak solution $(X^{(x,k)}, \La^{(x,k)})$ and the solution is unique in the sense of probability law.
}\end{Assumption}


 \begin{Assumption}\label{Assumption-Feller}{\rm
 There exist a positive constant $\delta_{0}  $ and an increasing and concave function $\varrho : [0,\infty) \mapsto [0,\infty)$ satisfying  \eqref{eq-varrho-conditions-Feller} such that  for all $R > 0$, there exists a constant $\kappa_{R} > 0$  such that
\begin{align}\label{eq-Feller-condition-q-kl}
   \sum_{l\in \ss\backslash\{k\}} \abs{q_{kl}(x) - q_{kl}(z)} \le \kappa_{R} \varrho(F( \abs{x-z})),
\   \text{ for all } k \in \ss \text{ and }  |x|\vee |z|  \le R \end{align}  where  $F(r): = \frac{r}{1+r}$ for $r \ge 0$, and either (i) or (ii) below holds:   \begin{itemize}
  \item[(i)] $d =1$. Then \eqref{eq1-1d path-cdn} and \eqref{eq1-1d-c-fn-path-cdn} hold.
  \item[(ii)]  \  $d \ge 2$. Then \begin{equation}\label{eq-Feller-condition-b-sg-c}\begin{aligned}
& \int_{U}  \bigl[ |c(x,k,u)-c(z,k,u)|^{2} \wedge (4 |x-z| \cdot |c(x,k,u) -c(z,k,u)|)\bigr] \nu(\d u) \\  &     \ \  +\,  2 \lan x-z, b(x,k)-b(z,k)\ran + |\sg(x,k)-\sg(z,k)|^{2 }  \le 2 \kappa_{R} |x-z| \varrho (|x-z|),
\end{aligned}\end{equation}
\end{itemize}
for all $k \in \ss$,  $x,z\in \R^{d}$  with  $ |x|\vee |z|  \le R$ and $|x-z| \le \delta_{0}$.
}\end{Assumption}

  \begin{thm}\label{thm-Feller}
Under  Assumptions \ref{Assumption-wk soln-well-posed} and \ref{Assumption-Feller}, the process $(X,\La)$ possesses the Feller property.
\end{thm}

\begin{rem}{\rm Feller and strong Feller properties for regime-switching (jump) diffusions have been investigated in \cite{Shao-15,XiZ-17,YZ-10}, among others. A standard assumption in these references is that the coefficients satisfy the Lipschitz condition.
In contrast, Theorem \ref{thm-Feller} establishes Feller property for  
system \eqref{eq:X} and \eqref{eq:La-SDE} under local non-Lipschitz conditions. When $d =1$, the result is even more remarkable. Indeed, Feller property is derived with only very mild conditions on $b(\cdot, k)$, $c(\cdot, k, u)$, and $Q(x)$, and with
virtually no condition imposed on  $\sigma(\cdot, k)$.
}\end{rem}

We will use the coupling method to prove Theorem \ref{thm-Feller}. To this end, let us first construct a coupling operator $\wdt \A$ for $\A$: For $f(x,i, z,j) \in C_{c}^{2} (\R^{d} \times \ss\times \R^{d} \times \ss )$, we define
\begin{equation}
\label{eq-A-coupling-operator} \begin{aligned}
\wdt{ \mathcal{A}} & f(x,i,z,j)  : =\! \bigl[ \wdt \Omega_{\text{d}}   + \wdt \Omega_{\text{j}}   + \wdt \Omega_{\text{s}} \bigr] f(x,i,z,j),
\end{aligned}\end{equation} where $ \wdt \Omega_{\text{d}}$,  $\wdt \Omega_{\text{j}}$, and    $ \wdt \Omega_{\text{s}}$ are defined as follows.
For $x,z\in \R^{d}  $ and $i,j\in \ss $, we set $a(x,i)= \sigma(x,i)\sigma(x,i)'$ and
$$\begin{aligned}a(x,i,z,j) & =\begin{pmatrix}
a(x,i) & \sigma (x,i) \sigma (z,j)' \\
\sigma (z,j) \sigma (x,i)' & a(z,j)
\end{pmatrix},\ \
b(x,i,z,j) =\begin{pmatrix}
b(x,i)\\
b(z,j) \end{pmatrix}. \end{aligned}$$  Then we define
\begin{equation}\label{eq-Omega-d-defn} \begin{aligned}
\wdt {\Omega}_{\text{d}}f(x,i,z,j): =\frac
{1}{2}\hbox{tr}\bigl(a(x,i,z,j)D^{2}f(x,i, z,j)\bigr)  +\langle
b(x,i,z,j), D f(x,i, z,j)\rangle,
\end{aligned}\end{equation}
 \begin{equation}\label{eq-Omega-j-defn}
\begin{aligned}   \displaystyle\wdt {\Omega}_{\text{j}} f(x,i,z,j)
& : = \int_{U}
\big[f(x+c(x,i,u),i, z+c(z,j,u), j)-f(x,i,z,j)  \\
  & \quad   -  \langle D_{x } f(x,i,z,j),   c(x,i,u)\rangle 
 - \langle D_{z} f(x,i, z,j),   c(z,j,u)\rangle   \big] \nu(\d u),
\end{aligned}
\end{equation} where $D f(x,i, z,j)= (D_{x } f(x,i,z,j), D_{z} f(x,i,z,j))'$ is the gradient and $D^{2}f(x,i,z,j)$   the Hessian matrix of   $f$ with respect to the $x,z$ variables,  and
\begin{equation}
\label{eq-Q(x)-coupling} \begin{aligned}
\wdt \Omega_{\text{s}}  f(x,i,z,j) & : =
 \sum_{l\in\ss}[q_{il}(x)-q_{jl}(z)]^+(   f(x,l, z, j)-  f(x,i, z, j))\\
&\quad +\sum_{l\in\ss}[q_{jl}(z) -q_{il}(x)]^+( f(x,i, z, l)- f(x,i, z, j) )\\
&\quad  +\sum_{l\in\ss}[ q_{il}(x) \wedge q_{jl}(z) ](  f(x,l, z, l)-f(x,i, z, j)).
\end{aligned}\end{equation} For convenience of later presentation, for 
any function $f : \R^{d}\times \R^{d}\mapsto \R$, let $\wdt f:\R^{d} \times \ss\times \R^{d} \times \ss \mapsto \R$ be defined by $\wdt f(x,i,z,j): =f(x,z)$.  Now we   denote $$\wdt \LL_{k} f(x,z) =( \wdt {\Omega}_{\text{d}}^{(k)} + \wdt {\Omega}_{\text{j}}^{(k)}   ) f(x,z) : = ( \wdt {\Omega}_{\text{d}} + \wdt {\Omega}_{\text{j}}   ) \wdt f(x,k,z,k), \forall f \in C^{2}_{c}(\R^{d}\times \R^{d}) $$ for each $k\in \ss$.
We proceed to establish the following lemma.

\begin{lem}\label{lem-Feller-operator-estimation}
 Suppose Assumption \ref{Assumption-Feller} holds. Consider the functions  \begin{equation}
\label{eq-f-g-functions} g(x,k,z,l): = \one_{\{k\neq l\}}, \text{ and }
  f(x,k, z,l): = F(|x-z|) + \one_{\{ k \neq l\}},  \
\end{equation} for $ (x,k, z, l) \in \R^{d}\times \ss \times \R^{d}\times \ss.$
    Then    we have
    \begin{equation}
 \label{eq1-switching-est}
  \wdt   \A g(x,k,z,l)   \le  \kappa_{R} \varrho(F( \abs{x-y})),  \text{ for all }k,l \in\ss \text{ and }x,z\in \R^{d} \text{ with  }|x| \vee |z| \le R,
\end{equation} and
 \begin{equation}
\label{eq-A-couple-est}
\wdt \A f(x, k,z, k) \le  2  \kappa_{R}\varrho (F (|x-z|)),
\end{equation}   for all $k\in\ss $  and $x,z\in \R^{d}$   with  $|x| \vee |z| \le R$   and  $0 < |x-z| \le \delta_{0}$; in which $\kappa_{R}$ is the same positive constant as in Assumption \ref{Assumption-Feller}.
  \end{lem}

  \begin{proof}
  Consider the function $g(x,k,z,l): = \one_{\{k\neq l\}}$. It follows directly from the definition that $\wdt   \A g(x,k,z,l) = \wdt \Omega_{\text{s}} g(x,k,z,l) \le 0$ when $k \neq l$. When $k =l$, we have  from \eqref{eq-Feller-condition-q-kl} that
\begin{align}\label{eq-coupling est 1}
 \nonumber  \wdt   \A g(x,k,z,l) & =\wdt \Omega_{\text{s}} g(x,k,z,k) & \\
 \nonumber  &   = \sum_{ i\in \ss} [q_{ki}(x)-q_{ki}(z)]^+( \one_{\{i \neq k \}} - \one_{\{ k\neq k\}})\\
\nonumber  &\qquad+\sum_{i \in\ss}[q_{ki}(z)  -q_{ki}(x)]^+(  \one_{\{i \neq k \}} - \one_{\{ k\neq k\}} ) + 0 \\
 & \le \sum_{i \in\ss, i \neq k} \abs{q_{ki}(x)-q_{ki}(z)}
 \le  \kappa_{R} \varrho(F( \abs{x-y})).
\end{align} Hence \eqref{eq1-switching-est} holds for all $k,l \in\ss$ and $x,z\in \R^{d}$ with  $|x| \vee |z| \le R$.
On the other hand, when $d\ge 2$,  \eqref{eq-Feller-condition-b-sg-c} and Lemma 4.5 of \cite{XiZ-18b} reveals  that  $$  \wdt \LL_{k}F(|x-z|) =  ( \wdt \Omega_{\text{d}}^{(k)} + \wdt \Omega_{\text{j}}^{(k)})  F(|x-z|) \le \kappa_{R} \varrho (F(|x-z|))$$ and hence \begin{equation}\label{eq-wdt-LL-F estimate}
\wdt \A f(x, k,z, k) =   \wdt \LL_{k}F(|x-z|)
+ \wdt \Omega_{\text{s}} g(x,k,z,k)  \le  2 \kappa_{R} \varrho (F(|x-z|)),
\end{equation} for all  $k \in \ss$,  $x,z\in \R^{d}$ with  $|x| \vee |z| \le R$ and $0 < |x-z| \le \delta_{0}$, where $\wdt\LL_{k}$ is the basic coupling operator for $\LL_{k}$ of \eqref{eq-Lk-operator-defn}.

We next show that \eqref{eq-A-couple-est} holds when $d =1$. Indeed, taking advantage of the fact that $d=1$, we see that 
\begin{align*}
&\wdt \Omega_{\text{d}}^{(k)} F(|x-z|) \\ &\ \ = F'(|x-z|) \sgn(x-z) (b(x,k) - b(z,k)) + F''(|x-z|) (\sigma(x,k)-\sigma(z,k))^{2}.
\end{align*}  But since $F'(r) = \frac{1}{(1+r)^{2}} $ and $F''(r)  =-\frac{2}{(1+r)^{3}} < 0$ for $r \ge 0$, we have from \eqref{eq1-1d path-cdn} that \begin{align}\label{eq-coupling est 2}
\nonumber \wdt \Omega_{\text{d}}^{(k)} F(|x-z|) & \le \frac{1}{(1+ |x-z|)^{2}}  \sgn(x-z) (b(x,k) - b(z,k))\\ &  \le \frac{\kappa_{R}\varrho(|x-z|)}{(1+ |x-z|)^{2}} \le \kappa_{R} \varrho(F(|x-z|)),
\end{align} for all $x,z\in \R$ with $|x| \vee |z| \le R$ and $0< |x-z| \le \delta_{0}$, where we used the first equation in \eqref{eq-varrho-conditions-Feller} to derive the last inequality.

On the other hand, since the function $F$ is concave on $[0,\infty)$, we have $F(r) -F(r_{0}) \le F'(r_{0})(r-r_{0})$ for all $r, r_{0} \in [0,\infty)$. Applying this inequality with $r_{0} = |x-z|$ and $r=|x-z+ c(x,k,u)- c(z,k,u)|$ yields
\begin{align*}
F&(|x-z+ c(x,k,u)- c(z,k,u)|)-F( |x-z|) \\&\le F'( |x-z|) ( |x-z+ c(x,k,u)- c(z,k,u)|-  |x-z|).
\end{align*} Furthermore, since by \eqref{eq1-1d-c-fn-path-cdn}, the function $x\mapsto x+ c(x,k,u)$ is increasing, it follows that for $x> z$ \begin{align*}
F&(|x-z+ c(x,k,u)- c(z,k,u)|)-F( |x-z|) \\&\le F'( |x-z|) ( x-z+ c(x,k,u)- c(z,k,u) -   (x-z)) \\
 & = F'( |x-z|) ( c(x,k,u)- c(z,k,u)).
\end{align*}
As a result, we can compute
\begin{align*}
  &  \wdt\Omega_{\text{j}}^{(k)}  F(|x-z|) \\
   & \ \ = \int_{U} [F(|x-z+ c(x,k,u)- c(z,k,u)|)-F( |x-z|) \\
   & \ \ \qquad- F'( |x-z|)\sgn(x-z) ( c(x,k,u)- c(z,k,u)) ]   \nu(\d u)  \\
    & \ \ \le  \int_{U}[ F'( |x-z|) ( c(x,k,u)- c(z,k,u)) -  F'( |x-z|) ( c(x,k,u)- c(z,k,u))]\nu(\d u) \\
    &\ \  =0,
\end{align*} for all $x > z$. By symmetry, we also have $\wdt\Omega_{\text{j}}^{(k)}  F(|x-z|) \le 0$ for $x < z$. These observations, together with \eqref{eq-coupling est 1} and \eqref{eq-coupling est 2}, imply  that $$\wdt \A f(x,k,z,k) =( \wdt \Omega_{\text{d}}^{(k)} + \wdt \Omega_{\text{j}}^{(k)} ) F(|x-z|) + \wdt \Omega_{\text{s}} g(x,k,z,k)  \le 2 \kappa_{R} \varrho(F(|x-z|)),$$ for all $k\in \ss$, $x,z\in \R$ with $|x| \vee |z| \le R$ and $0< |x-z| \le \delta_{0}$
This completes the proof. \qed
\end{proof}

  \begin{proof}[Proof of Theorem \ref{thm-Feller}]  It is straightforward to verify that the function $f$ of \eqref{eq-f-g-functions} defines a bounded  metric on $\R^{d}\times \ss$.  Let
$(\wdt X(\cdot), \wdt \La(\cdot),   \wdt Z(\cdot), \wdt \Xi(\cdot))$ denote the coupling process
corresponding to the coupling operator  $\wdt \A$
with initial condition $(x, k,z, k)$,  in which $\delta_{0} > |x-z | > 0$.  Define  $\zeta: = \inf\{ t \ge0: \wdt \La(t) \neq \wdt \Xi(t)\}  $. Note that $\P\{ \zeta > 0\} =1$.  Suppose $|x-z| > \frac{1}{n_{0}}$ for some $n_{0}\in \mathbb N$.
         For $n\ge n_{0}$ and $  R >  |x| \vee |z|$,  define
\begin{align*}
 & T_n : =   \inf \Bigl\{  t \ge 0:   | \wdt X  (t) - \wdt Z  (t) | < \frac{1}{n}\Bigr\}, \\
    &  \tau_{R} : =  \inf \{  t \ge 0:   | \wdt X (t) | \vee |\wdt  Z  (t) | \vee \wdt \La(t) \vee \wdt\Xi(t)> R\}, \end{align*}
    and \begin{align*}
    & S_{\delta_{0}}: = \inf\{t \ge 0:  |\wdt  X  (t) - \wdt  Z  (t) | > \delta_{0} \}.
\end{align*} We have $\tau_{R}\to \infty$ and $T_{n}\to T$ a.s. as $R\to \infty$ and $n \to \infty$, respectively, in which $T$ denotes the first time when $\wdt X(t) $ and $   \wdt Z(t)$ coalesce.
 To simplify notation,
 denote $\wdt \Delta(s) : = \wdt X(s) - \wdt Z(s)$. By It\^o's formula and \eqref{eq-wdt-LL-F estimate}, we have
\begin{align*}
 \E&[F(|\wdt  \Delta(t\wedge T_{n} \wedge S_{\delta_{0}} \wedge \tau_{R}\wedge \zeta)|)] \\
 & \le \E \big[f\big(\wdt X(t\wedge T_{n} \wedge S_{\delta_{0}} \wedge \tau_{R}\wedge \zeta), \wdt \La(t\wedge T_{n} \wedge S_{\delta_{0}} \wedge \tau_{R}\wedge \zeta), \\
    & \qquad \qquad \wdt Z(t\wedge T_{n} \wedge S_{\delta_{0}} \wedge \tau_{R}\wedge \zeta), \wdt\Xi(t\wedge T_{n} \wedge S_{\delta_{0}} \wedge \tau_{R}\wedge \zeta)  \big)\big] \\
   & = F(|x-z|) + \E\biggl[\int_{0}^{t\wedge T_{n} \wedge S_{\delta_{0}} \wedge \tau_{R}\wedge \zeta} \wdt \A f(\wdt X(s), \Delta(s), \wdt Z(s),\wdt\Xi(s)) \d s \biggr]\\
   & \le  F(|x-z|) +2  \kappa_{R}  \E\biggl[\int_{0}^{t\wedge T_{n} \wedge S_{\delta_{0}} \wedge \tau_{R}\wedge \zeta} \varrho  (F(|\wdt  \Delta(s) |))\d s \biggr].
   \end{align*}
   Now passing to the limit as $n \to \infty$, it follows from the bounded and monotone convergence theorems that
  \begin{align*}
\E& [F(|\wdt  \Delta (t\wedge T \wedge S_{\delta_{0}}\wedge \tau_{R} \wedge \zeta)|)] \\
& \le  F(|x-z|) + 2\kappa_{R}  \E\biggl[\int_{0}^{t\wedge T \wedge S_{\delta_{0}} \wedge \tau_{R}\wedge \zeta} \varrho  (F(|\wdt  \Delta(s) |))\d s \biggr]\\
   & \le F(|x-z|) + 2\kappa_{R}  \E\biggl[\int_{0}^{t} \varrho  (F(|\wdt  \Delta (s\wedge T \wedge S_{\delta_{0}} \wedge \tau_{R} \wedge \zeta)|))\d s \biggr]\\
   &  \le  F(|x-z|) +2 \kappa_{R}\int_{0}^{t} \varrho \bigl(\E[F(|\wdt  \Delta (s\wedge T \wedge S_{\delta_{0}}\wedge \tau_{R}\wedge \zeta) |)]\bigr)\d s,
   \end{align*}
 where we used the concavity of $\varrho $ and Jensen's inequality to obtain the last inequality.
 Then
 using 
 Bihari's inequality, we have
\begin{displaymath}
\E [F(|\wdt  \Delta (t\wedge T \wedge S_{\delta_{0}}\wedge \tau_{R} \wedge \zeta)|)]  \le G^{-1} (G\circ F(|x-z|) + 2\kappa_{R} t),
\end{displaymath}  where the function $G(r): = \int_{1}^{r} \frac{\d s} {\varrho (s)}$ is strictly increasing and satisfies $G(r) \to -\infty$ as $r \downarrow 0$.  In addition, since the function $F$ is strictly increasing, we have
\begin{align*}
   F(\delta_{0}) \P\{S_{\delta_{0}} < t \wedge T\wedge \tau_{R}\wedge \zeta \}
    & \le  \E [F(|\wdt  \Delta (t\wedge T \wedge S_{\delta_{0}} \wedge \tau_{R}\wedge \zeta) |) \one_{\{ S_{\delta_{0}} < t \wedge T\wedge \tau_{R}\wedge \zeta\}}]    \\
    &   \le \E [F(|\wdt  \Delta (t\wedge T \wedge S_{\delta_{0}} \wedge \tau_{R}\wedge \zeta) |)]  \\
    & \le G^{-1} (G\circ F(|x-z|) +2 \kappa_{R} t).
\end{align*} This implies that
\begin{align*}
 \P\{S_{\delta_{0}} < t \wedge T\wedge \tau_{R}\wedge \zeta \} &\le \frac{G^{-1} (G\circ F(|x-z|) + 2\kappa_{R} t) }{F(\delta_{0})}\\ &= \frac{1+ \delta_{0}}{\delta_{0}}G^{-1} (G\circ F(|x-z|) + 2\kappa_{R} t).
\end{align*}

 For any $t \ge 0$ and $\e > 0$, since $\lim_{R\to\infty} \tau_{R} = \infty$ a.s., we can choose
 $R > 0$ sufficiently large so that \begin{equation}
\label{eq-tau-R >t}
\P(t\wedge\zeta > \tau_{R} )\le \P(t > \tau_{R}) < \e. \end{equation}Then it follows that
\begin{align}
\nonumber  & \E [F(|\wdt \Delta(t\wedge \zeta)|)] \\
  \nonumber & \ =   \E [F(|\wdt  \Delta(t \wedge \zeta\wedge \tau_{R}) |)\one_{\{t \wedge \zeta\le \tau_{R} \}}] +  \E [F(|\wdt  \Delta(t\wedge \zeta) |) \one_{\{ t \wedge \zeta> \tau_{R}\}}]  \\
 \nonumber   &\ \le   \E[F(|\wdt  \Delta (t \wedge \zeta \wedge T\wedge \tau_{R}) |)]   + \e   \\
 \nonumber   & \  = \E[F(|\wdt  \Delta (t\wedge \zeta\wedge T\wedge \tau_{R}) |)\one_{\{ S_{\delta_{0}} < t \wedge T\wedge \tau_{R}\wedge \zeta\}}]  \\ \nonumber &\quad
    + \E[F(|\wdt  \Delta (t\wedge \zeta\wedge T\wedge \tau_{R}) |)\one_{\{ S_{\delta_{0}} \ge  t \wedge T\wedge \tau_{R}\wedge \zeta\}}]  + \e\\
 \nonumber   &\ \le \P\{S_{\delta_{0}} < t \wedge T\wedge \tau_{R}\wedge \zeta \} + \E [F(|\wdt  \Delta (t\wedge T\wedge \tau_{R}\wedge S_{\delta_{0}}\wedge \zeta )|)] + \e\\
 \label{eq0-F-Delta-mean-est}   &\ \le \frac{1 +2 \delta_{0}}{\delta_{0}} G^{-1} (G\circ F(|x-z|) + 2\kappa_{R} t) + \e.
\end{align}
Passing to the limit, we obtain \begin{equation}
\label{eq1-F-Delta-mean-est}
\lim_{x-z \to 0} \E [F(|\wdt \Delta(t\wedge \zeta) |)] \le 0+ \e =\e.
\end{equation}   Since $\e >0$ is arbitrary, it follows that  $\lim_{x-z \to 0}\E[F(|\wdt X(t\wedge \zeta) - \wdt Z(t\wedge \zeta)| ) ]  = 0$.

Choose $R> 0$ as in \eqref{eq-tau-R >t}. Then we use  \eqref{eq1-switching-est} and \eqref{eq0-F-Delta-mean-est} to compute
\begin{align*}
  \P     \{\zeta \le t \} &=    \P     \{\zeta \le t, \tau_{R} < t \} + \P     \{\zeta \le t, \tau_{R} \ge  t \}   \\
    &   \le  \P     \{  \tau_{R} < t \}  + \E\big[\one_{\{\wdt \La(t \wedge \zeta \wedge \tau_{R}) \neq\wdt \Xi(t \wedge \zeta \wedge \tau_{R})  \}} \big] \\
    & < \e + \E [g(\wdt X(t \wedge \zeta \wedge \tau_{R}), \wdt \La(t \wedge \zeta \wedge \tau_{R}), \wdt Z(t \wedge \zeta \wedge \tau_{R}), \wdt \Xi(t \wedge \zeta \wedge \tau_{R}))] \\
    & = \e + \E\bigg[\int_{0}^{t \wedge \zeta \wedge \tau_{R}} \wdt \A g(\wdt X(s), \wdt \La(s), \wdt Z(s), \wdt\Xi(s))\d s\bigg] \\
    &  \le  \e + \E\bigg[\int_{0}^{t \wedge \zeta \wedge \tau_{R}} \kappa_{R} \varrho (F(|\wdt \Delta(s)|)) \d s \bigg]
   \\&  \le \e+ \E\bigg[\int_{0}^{t  \wedge \tau_{R}} \kappa_{R} \varrho (F(|\wdt \Delta(s\wedge \zeta)|)) \d s \bigg]\\
       & \le \e+ \E\bigg[\int_{0}^{t } \kappa_{R} \varrho (F(|\wdt \Delta(s\wedge \zeta)|)) \d s \bigg]
      \\&  \le \e + \kappa_{R}  \int_{0}^{t }   \varrho ( \E[F(|\wdt \Delta(s\wedge \zeta)|)]) \d s \\
       & \le \e +  \kappa_{R}  \int_{0}^{t }   \varrho \bigg(\frac{1 +2 \delta_{0}}{\delta_{0}} G^{-1} (G\circ F(|x-z|) + 2\kappa_{R} s) + \e\bigg) \d s \\
       & \le \e + \kappa_{R} t   \varrho \bigg(\frac{1 +2 \delta_{0}}{\delta_{0}} G^{-1} (G\circ F(|x-z|) +2 \kappa_{R} t) + \e\bigg).
    \end{align*}
  Passing to the limit as $x-z \to 0$, we obtain \begin{equation}
\label{eq2-P-zeta< t-est}
 \limsup_{x-z\to 0}\P     \{\zeta \le t \} \le \e + \kappa_{R} t \varrho(\e).
\end{equation}

 Finally, we combine  \eqref{eq1-F-Delta-mean-est} and \eqref{eq2-P-zeta< t-est} to obtain
\begin{align*}
\E& [f(\wdt X(t), \wdt \La(t),   \wdt Z(t), \wdt \Xi(t) )] \\&= \E \bigl[F(  | \wdt X(t  )-   \wdt Z(t  )  |) +  \one_{\{\wdt \La(t) \neq \wdt\Xi(t) \}}  \bigr] \\
 & =  \E \bigl[F(  | \wdt X(t  )-   \wdt Z(t  )  |)\one_{\{\zeta > t \}} + F(  | \wdt X(t  )-   \wdt Z(t  )  |)\one_{\{\zeta \le  t \}}  + \one_{\{\wdt \La(t) \neq \wdt\Xi(t) \}}  \bigr]  \\
  & \le\E \bigl[F(  | \wdt X(t \wedge\zeta )-   \wdt Z(t  \wedge\zeta)  |)\big] + 2\P\{\zeta \le t \} \\ &
   \to \e +2( \e + \kappa_{R} t \varrho(\e)), \text{ as } |x-z| \to 0.
\end{align*} Since $\e > 0$ is arbitrary and $\lim_{r\downarrow 0}\varrho(r) =0$, it  follows that $$\lim_{x-z \to 0 } \E [f(\wdt X(t), \wdt \La(t),   \wdt Z(t), \wdt \Xi(t) )] =0.$$  Recall  that $f$ is a bounded metric on $\R^{d}\times \ss$.
  Hence it follows that $$W_{f}(P(t,x,k,\cdot), P(t,z,k,\cdot) )\le  \E [f(\wdt X(t), \wdt \La(t),   \wdt Z(t), \wdt \Xi(t) )]  \to 0 \text{ as } x\to z,$$ where  for two probability measures $\mu$ and $\nu$ on $\R^{d}\times\ss$, the Wasserstein distance $W_{f}(\mu,\nu)$ is defined as $$W_{f}(\mu,\nu):= \inf\bigg\{\sum_{i,j\in\ss}\int f(x,i,y,j) \pi(\d x,i,\d y,j), \pi \in \mathcal{C}(\mu,\nu) \bigg\},$$  here $\mathcal{C}(\mu,\nu) $ is the collection of coupling measures for $\mu$ and $\nu$. Therefore the desired Feller property follows  from Theorem 5.6 of \cite{Chen04}.\qed
\end{proof}

\section{Strong Feller Property}\label{sect-str-Feller}




\begin{Assumption}\label{Assumption-martingale-well-posed}{\rm
For each $k\in \ss$ and   $x\in \R^{d}$, the stochastic differential equation \eqref{eq:X-k} has a non-exploding weak solution $X^{(k)}$ with initial condition  $x$ and the solution is unique in the sense of probability law.
}\end{Assumption}

\begin{Assumption}\label{Assump-Uniform-Elliptic}{\rm
The process $X^{(k)}$ is strong Feller.
} \end{Assumption}

\begin{Assumption}\label{assumption-q-bdd}{\rm
Assume that
\begin{equation}
\label{eq-q-bdd}
H: =\sup\{q_{k}(x): x\in \R^{d}, k\in \ss \}  < \infty,
\end{equation} and that
  there exists a positive constant  $\kappa $ such that
\begin{equation}
\label{eq-q-ratio-bdd}
0\le    q_{kl}(x)  \le \kappa
l 3^{-l} \text{ for all }x\in \R^{d} \text{ and  }k \neq l \in \ss.
\end{equation}
}\end{Assumption}

Let us briefly comment on the above assumptions. The existence and uniqueness of weak solution to \eqref{eq:X-k} is related to the study of martingale problem for L\'evy type operators; see, for example, \cite{Koma-73} and \cite{Stroock-75}. Condition \eqref{eq-q-bdd} in Assumption \ref{Assump-Uniform-Elliptic} is stronger than \eqref{q_k<Hk} in Assumption \ref{assumption-non-linear-growth}. We need such a uniform bound in  \eqref{eq-q-bdd} so that we can establish the series representation for the resolvent
 of the regime-switching jump diffusion $(X,\La)$ in Lemma \ref{lem-resolvant-series}, which, in turn, helps to establish the strong Feller property for $(X,\La)$. In general one can obtain the strong Feller property for $X^{(k)}$ under suitable non-degenerate conditions (\cite{Kunita-13}) and certain regularity conditions such as (local) Lipschitz conditions of the coefficients.
The following non-Lipschitz sufficient condition for strong Feller property   was established in \cite{XiZ-18b}.

\begin{lem}\label{lem-strong-Feller} 
Suppose that Assumptions \ref{Assumption-martingale-well-posed} holds. In addition, for any given $k\in \ss$,
suppose that for each $R > 0$, there exist positive constants $\lambda_{R}$ and $\kappa_{R}$ such that for all $x, z\in \R^{d}$ with $|x| \vee |z| \le R$, we have \begin{align*} & \lan \xi, a(x,k)\xi\ran \ge \lambda_{R}|\xi |^{2}, \quad   \forall  \xi \in \R,\end{align*} and \begin{align*}
&\int_{U}  \bigl[ |c(x,k,u)-c(z,k,u)|^{2} \wedge (4 |x-z| \cdot |c(x,k,u) -c(z,k,u)|)\bigr] \nu(\d u) \\  &\  + 2  \lan x-z, b(x,k)-b(z,k) \ran + |\sg_{\lambda_{R}}(x,k)-\sg_{\la_{R}}(z,k)|^{2 }     \le 2 \kappa_{R} |x-z| \vartheta(|x-z|) \end{align*} whenever $|x-z| \le \delta_{0}$, where $\delta_{0}$ is a positive constant, $\vartheta$ is a nonnegative function   defined on $[0,\delta_{0}]$ satisfying $\lim_{r\to 0} \vartheta (r) =0$, and $\sg_{\lambda_{R}}(x,k)$ is the unique symmetric nonnegative definite matrix-valued function such that $\sg_{\lambda_{R}}(x,k)^{2} = a(x,k)- \lambda_{R}I$. Then the process  $X^{(k)} $  of  \eqref{eq:X-k} is strong Feller continuous.
\end{lem}

Next for each $(x,k) \in \R^{d}\times\ss$, as in  \cite[Section 8.2]{Oksendal},   we kill the
  process $X^{(k)}$ at rate $(-q_{kk})$:
\begin{equation}\label{eq-killed process}\begin{aligned}
\E_{k}[f(\wdt{X}^{(k)}_{x}(t))]&  =\disp
\E_{k}\biggl[f(X_{x}^{(k)}(t))\exp
\biggl\{\int_{0}^{t}q_{kk}(X^{(k)}_{x}(s)){\d}s\biggr\}\biggr]\\ & =\disp
\E^{(x,k)}[t<\tau; f(X^{(k)}(t))], \quad f\in \B_{b}(\R^{d}),
\end{aligned}\end{equation} to get a subprocess $\wdt{X}^{(k)}$, where $\tau:=\inf\{t\ge 0: \La(t) \not=\La(0)\}$. Equivalently,
$\wdt{X}^{(k)}$ can be defined as $\wdt{X}^{(k)}(t)=X^{(k)}(t)$ if $t<\tau$
and $\wdt{X}^{(k)}(t)=\partial$ if $t\ge \tau$, where
$  \partial$ is a
cemetery point or a coffin state
 added to $\R^{d}$ as in  \cite[p. 145]
 {Oksendal}.
 Note that in the above,  to get the killed process $\wdt{X}^{(k)}$ from the original process $X^{(k)}$, the killing rate is just the jumping rate of $\La$ from state $k$. Namely, the killing time is just the first switching time $\tau$.
To proceed, we denote the transition probability families of the process $X^{(k)}$ and the killed process
$\wdt{X}^{(k)}$ by $\{P^{(k)}(t,x,A): t \ge 0, x \in \R ^{d}, A \in
\B(\R ^{d})\}$ and $\{\wdt{P}^{(k)}(t,x,A): t \ge 0, x \in \R ^{d}, A \in
\B(\R ^{d})\}$,  respectively.


\begin{lem} \label{Lem-killed process str Feller}
Under  Assumptions \ref{Assumption-martingale-well-posed},  \ref{Assump-Uniform-Elliptic}, and \ref{assumption-q-bdd}, for each $k \in \ss$, the killed process $\wdt{X}^{(k)}$
has strong Feller property.
\end{lem}

\begin{proof} Let $\{P^{(k)}_{t}\}$ and $\{\wdt{P}^{(k)}_{t}\}$ denote the transition semigroups of $X^{(k)}$ and $\wdt{X}^{(k)}$, respectively. To prove the strong Feller property $\wdt{X}^{(k)}$, we need only
prove that for any given bounded measurable function $f$ on $\R ^{d}$, $\wdt{P}^{(k)}_{t}f(z)$ is continuous with respect to $z$ for all $t>0$. To this end, for fixed $t>0$ and $0<s<t$, set $g_{s}(z):=\wdt{P}^{(k)}_{t-s}f(z)$. Clearly, the function $g_{s}(\cdot)$ is bounded and measurable, see the Corollary to Theorem 1.1 in \cite{ChungZ-95}. By the strong Feller property of $X^{(k)}$,
$P^{(k)}_{s}g_{s} \in C_{b}(\R ^{d})$.

To proceed, by the Markov property, we have that
\begin{equation}\label{Ptildef}
\begin{aligned}
\wdt{P}^{(k)}_{t}f(x) &  =\E_{k}^{(x)}\biggl[f(X^{(k)}(t))\exp
\biggl\{\int_{0}^{t}q_{kk}(X^{(k)}(u)){\d}u\biggr\}\biggr]\\
&  =\E_{k}^{(x)}\biggl[\exp
\biggl\{\int_{0}^{s}q_{kk}(X^{(k)}(u)){\d}u\biggr\}\\
&  \qquad \times\E_{k}^{(X^{(k)}(s))}\biggl[f(X^{(k)}(t-s))\exp
\biggl\{\int_{0}^{t-s}q_{kk}(X^{(k)}(u)){\d}u\biggr\}\biggr]\biggr].
\end{aligned}\end{equation}
Meanwhile, we also have that \begin{equation}\label{PPtildef}
\begin{aligned}  P_{s}^{(k)}g_{s} (x) & =
P_{s}^{(k)}\wdt{P}^{(k)}_{t-s}f(x)  =\E_{k}^{(x)}\biggl[\wdt{P}^{(k)}_{t-s}f(X^{(k)}(s))\biggr]\\
&  =\E_{k}^{(x)}\biggl[\E_{k}^{(X^{(k)}(s))}\biggl[f(X^{(k)}(t-s))\exp
\biggl\{\int_{0}^{t-s}q_{kk}(X^{(k)}(u)){\d}u\biggr\}\biggr]\biggr].
 \end{aligned}\end{equation} Recall from Assumption \ref{Assump-Uniform-Elliptic} that $+\infty>H\ge -\inf\{q_{kk}(x): (x,k)\in \R^{2d}\times \ss\}$
   and $q_{kk}(x)\le 0$, and so \begin{equation}\label{e-Hs} 0\le 1-\exp
\biggl\{\int_{0}^{s}q_{kk}(X^{(k)}(u)){\d}u\biggr\}\le \bigl(1-e^{-Hs}\bigr).\end{equation} Thus, it follows from (\ref{Ptildef}), (\ref{PPtildef}) and (\ref{e-Hs}) that \begin{equation}\label{killedprocess-sfeller}
|P^{(k)}_{s}g_{s}(x)-\wdt{P}^{(k)}_{t}f(x)|\le \bigl(1-e^{-Hs}\bigr)\|f\| \to 0 \,\, \hbox{uniformly as} \,\, s\to 0,\end{equation} where $\|\cdot\|$ denotes the uniform (or supremum) norm. Combining this with the fact that $P^{(k)}_{s}g_{s} \in C_{b}(\R ^{d})$ implies that $\wdt{P}^{(k)}_{t}f\in C_{b}(\R ^{d})$, and so the desired strong Feller property follows.
\qed\end{proof}

The following lemma was proved in \cite{XiZ-18}.

\begin{lem} \label{lem-CP1}
Let $\Xi$ be a right continuous strong Markov process  and $q: \R^{d}\mapsto \R$ a nonnegative
bounded measurable function. Denote by $\wdt{\Xi}$   the subprocess of $\Xi$ killed at   rate $q$ with
lifetime $\zeta$:
\begin{equation}\label{kp2}
\E[f(\wdt{\Xi}^{(z)}({t}))] :=\disp \E\bigl[ t<\zeta;f(\Xi^{(z)}(t))\bigr]
=\E\biggl[f(\Xi^{(z)}(t))\exp\biggl\{-\int_{0}^{t}q(\Xi^{(z)}(s)){\d}s
\biggr\}\biggr].
\end{equation}
Then for any constant $\al >0$ and   nonnegative function  $\phi$ on $\R^{d}$, we have
\begin{equation}\label{eq:zeta}
\E[e^{-\al \zeta}\phi(\wdt{\Xi}^{(z)}(\zeta-))]=G_{\al}^{\wdt{\Xi}}(q\phi)(z),
\end{equation} where $\{G_{\al}^{\wdt{\Xi}}, \al >0\}$ denotes the resolvent
for the killed process $\wdt{\Xi}$.
\end{lem}


  For each $k \in \ss$, let $\{\wdt{G}^{(k)}_\al, \al>0\}$
be the resolvent for the generator $\LL_k+q_{kk}$.
Denote by $\{G_{\alpha}, \alpha >0 \}$
the resolvent for the generator $\A$ defined in \eqref{eq-operator}.
  Let
\begin{displaymath}
\wdt{G}_{\al}=\left(\begin{array}{cccc}
\wdt{G}^{(1)}_{\al} & 0 & 0 & \dots\\
0 & \wdt{G}^{(2)}_{\al} & 0 & \dots\\
0 & 0 & \wdt{G}^{(3)}_{\al} & \dots\\
\vdots & \vdots & \vdots & \ddots
\end{array} \right) \ \hbox{ and } \  Q^{0}(x)=Q(x)-\begin{pmatrix}
q_{11}(x) & 0 & 0 & \dots\\
0 & q_{22}(x) & 0 & \dots\\
0 & 0 & q_{33}(x) & \dots\\
\vdots & \vdots & \vdots & \ddots
\end{pmatrix} .
\end{displaymath}

\begin{lem} \label{lem-resolvant-series}
Suppose that  Assumptions \ref{Assumption-martingale-well-posed},  \ref{Assump-Uniform-Elliptic}, and \ref{assumption-q-bdd} hold. Then there exists a constant
${\al}_1>0$ such that for any ${\al}\ge {\al}_1$ and any
$f(\cdot,k)\in \B_{b}(\R^{d})$ with $k \in \ss$,
\begin{equation}\label{(FP17)}
G_{\al}f=\wdt{G}_{\al}
f+\sum_{m=1}^{\infty}\wdt{G}_{\al}\bigl(Q^{0}\wdt{G}_{\al}\bigr)^{m}f.
\end{equation}
\end{lem}

\begin{proof}
 Let $f(z,k)\geq
0$ on $\R ^{2d} \times \ss$. Applying the strong Markov property at
the first switching time $\tau$ and recalling the construction of
$(Z,\La)$, we obtain
 \begin{align*} \disp
G_\al f(z,k) & =\E_{z,k} \biggl[\int_0^\infty e^{-{\al} t}
f(Z(t),\La(t)){\d}t \biggr]\\
& =\E_{z,k} \biggl[\int_0^\tau e^{-{\al} t} f(Z(t), k){\d}t \biggr]+
\E_{z,k} \biggl[\int_\tau^\infty e^{-{\al} t}
f(Z(t), \La(t)){\d}t \biggr]\\
& =\wdt{G}^{(k)}_{\al} f(z,k)+\E_{z,k} \biggl[ e^{-{\al} \tau}
G_{\al} f(Z(\tau), \La(\tau))\biggr]\\
& =\wdt{G}^{(k)}_{\al}
f(z,k)+\sum_{l \in \ss\setminus \{k\}}\E_{z,k}
\biggl[
e^{-{\al} \tau} 
\biggl(-\,\frac{q_{kl}}{q_{kk}}\biggr)(Z({\tau-}))
G_{\al} f(Z({\tau-}),l)\biggr]\\
&  =\wdt{G}^{(k)}_{\al} f(z,k)+\sum_{l \in \ss\setminus \{k\}}
\wdt{G}_{\al}^{(k)} (q_{kl} G_{\al} f(\cdot, l))(z),
\end{align*}
where  the last equality follows from
(\ref{eq:zeta})  in Lemma~\ref{lem-CP1}.
Hence we have
\begin{equation}\label{(FP18)}
G_{\al}f(z,k)=\wdt{G}^{(k)}_{\al}f(\cdot,k)(z)
+\wdt{G}^{(k)}_{\al}\Biggl(\sum_{l \in \ss \setminus
\{k\}}q_{kl}G_{\al}f(\cdot,l)\Biggr)(z).
\end{equation} Repeating the above argument,
the second term on
the right-hand side of (\ref{(FP18)}) equals
$$\wdt{G}^{(k)}_{\al}\Biggl(\sum_{l \in \ss \setminus
\{k\}}q_{kl}\wdt{G}^{(l)}_{\al}f(\cdot,l)\Biggr)(z)+
\wdt{G}^{(k)}_{\al}\Biggl(\sum_{l \in \ss \setminus
\{k\}}q_{kl}\wdt{G}^{(l)}_{\al}\Biggl(\sum_{l_{1} \in \ss \setminus
\{l\}}q_{ll_{1}}G_{\al}f(\cdot,l_{1})\Biggr)\Biggr)(z).$$ Hence, we
further obtain that for any fixed $k \in \ss$ and any integer $m\ge
1$,
\begin{equation}\label{(FP19)}
G_{\al}f(z,k)=\sum_{i=0}^{m} \psi^{(k)}_{i}(z)+R^{(k)}_{m}(z),
\end{equation} where
 \begin{align*}
&   \psi^{(k)}_{0}=\wdt{G}^{(k)}_{\al}f(\cdot,k), \\
 &  \psi^{(k)}_{1}=\wdt{G}^{(k)}_{\al}\Biggl(\sum_{l \in \ss \setminus
\{k\}}q_{kl}\wdt{G}^{(l)}_{\al}f(\cdot,l)\Biggr)=\wdt{G}^{(k)}_{\al}\Biggl(\sum_{l
\in \ss \setminus \{k\}}q_{kl}\psi^{(l)}_{0}\Biggr), \\
&  \psi^{(k)}_{i}=\wdt{G}^{(k)}_{\al}\Biggl(\sum_{l
\in \ss \setminus \{k\}}q_{kl}\psi^{(l)}_{i-1}\Biggr)\quad \hbox{for} \quad i\ge 1, \intertext{and}
& R_{m}^{(k)} = \wdt{G}^{(k)}_{\al}\Biggl(\sum_{l_{1} \in \ss\setminus\{k \}} q_{k,l_{1}} \wdt{G}^{(l_{1})}_{\al}
   \Biggl(\sum_{l_{2}\in \ss \setminus \{l_{1}\}}q_{l_{1}, l_{2}}\wdt{G}^{(l_{2})}_{\al}
    \Biggl( \dots \Biggl( \sum_{l_{m-1} \in \ss\setminus\{l_{m-2}\}} q_{l_{m-2}, l_{m-1}}  \\
    & \qquad \qquad\qquad\qquad\qquad\qquad \wdt{G}^{(l_{m-1})}_{\al} \Biggl( \sum_{l_{m} \in \ss\setminus\{l_{m-1} \}} q_{l_{m-1}, l_{m}}G_{\al} f (\cdot, l_{m}) \Biggr)\Biggr)\Biggr)\Biggr) \Biggr).
\end{align*}
We have
\begin{equation}\label{eq-psi0}\|\psi_{0}^{(k)}\|=\biggl\|\E_{\cdot,k} \biggl[\int_0^\tau e^{-{\al} t} f(Z(t), k){\d}t \biggr]\biggr\|\le \frac{\|f\|}{\al}.\end{equation} Note that the same calculation reveals that  \eqref{eq-psi0} in fact holds for all $l \in \ss$, $\|\psi_{0}^{(l)}\|\le \frac{\|f\|}{\al}.$
Thanks to Assumption  \ref{assumption-q-bdd},   $q_{kl}(z) \le   \frac{\kappa l}{3^{l}}$  for all  $l \neq k$  and $x \in \R^{d}$.  Consequently, we can compute \begin{equation}\label{eq-psi1} \begin{aligned}\|\psi^{(k)}_{1}\| \le \sum_{l \in \ss \setminus
\{k\}}\|\wdt{G}^{(k)}_{\al}(q_{kl}\psi^{(l)}_{0})\| & \le
\sum_{l \in \ss \setminus\{k\}} \frac{\kappa l}{3^{l}} \cdot\frac{ \|\psi^{(l)}_{0}\| }{\alpha} \\ & \le \sum_{l \in \ss \setminus\{k\}} \frac{\kappa l}{3^{l}} \cdot\frac{  \|f\|  }{\alpha^{2}} =  \frac{3 \kappa}{4\alpha}\cdot \frac{\|f\|}{\al}.\end{aligned}\end{equation} As before, we observe that \eqref{eq-psi1} actually holds for all $l \in \ss$. Similarly,
we can use induction to show that \begin{equation}\label{eq-psi-i}
\|\psi^{(k)}_{i}\|\le  \biggl(\frac{3\kappa}{4\alpha}\biggr)^{i} \cdot \frac{\|f\|}{\alpha}, \qquad  \text{ for   }i \ge 2,\end{equation} and  \begin{equation}\label{eq-R_m^k-estimate}
\|R^{(k)}_{m}\|\le     \biggl(\frac{3\kappa}{4\alpha}\biggr)^{m+1}\cdot\frac{\|f\|}{\al}.\end{equation} Now let $\al_{1}: = \frac{3\kappa+1}{4}$ and $\al \ge \al_{1}$. Then we have for each $k \in \ss$,
$G_{\al}f(\cdot,k)=\sum_{i=0}^{\infty} \psi^{(k)}_{i}$, which
clearly implies (\ref{(FP17)}).  The lemma is proved. \qed \end{proof}

Lemma \ref{lem-resolvant-series} establishes an explicit  relationship of the resolvents for $(Z,\La)$ and the killed processes $\wdt{Z}^{(k)}$, $k\in \ss$. This, together with the strong Feller property for the killed processes $\wdt{Z}^{(k)}$, $k\in \ss$ (Lemma \ref{Lem-killed process str Feller}), enables us to  derive the strong Feller property for $(Z,\La)$ in the following theorem.

\begin{thm} \label{thm-sFeller} Suppose that  Assumptions \ref{Assumption-wk soln-well-posed}, \ref{Assumption-martingale-well-posed},  \ref{Assump-Uniform-Elliptic}, and \ref{assumption-q-bdd} hold.  Then the process $(X,\La)$ has the strong Feller property.
\end{thm}

\begin{proof} The proof is almost identical to that of Theorem 5.4 in \cite{XiZ-18} and for brevity, we shall only give a sketch here. Denote the transition probability family of Markov process $(X,\La)$ by $\{P(t,(x,k),A): t\ge 0,(x,k)\in \R^{d} \times \ss, A\in {\cal B}(\R ^{d} \times \ss)\}$. Then it follows from Lemma \ref{lem-resolvant-series} that
\begin{equation}\label{(FP22)}\begin{aligned} &  P(t,(x,k),A\times \{l\})=\delta_{kl} \wdt{P}^{(k)}(t,x,A)\\ & \ \
   +\sum_{m=1}^{+\infty} \  \idotsint\limits_{0<t_{1}<\cdots
<t_{m}<t}
  \sum_{{l_{1}\in \ss\setminus\{l_{0}\}, l_{2}\in \ss\setminus\{l_{1}\}, \dots, l_{m}\in
\ss\setminus\{l_{m-1}\},}\atop{l_{0}=k, \, l_{m}=l}}\int_{\R ^{d}} \cdots
\int_{\R ^{d}}  \wdt{P}^{(l_{0})}(t_{1},x,{\d}x_{1})\\
&   \ \ \times q_{l_{0}l_{1}}(x_{1}) \wdt{P}^{(l_{1})}(t_{2}-t_{1},x_{1},{\d}x_{2})\cdots
q_{l_{m-1}l_{m}}(x_{m})\wdt{P}^{(l_{m})}(t-t_{m},x_{m},A) {\d}t_{1} {\d}t_{2}
\dots {\d}t_{m},\end{aligned}\end{equation} where $\delta_{kl}$ is the Kronecker
symbol in $k$, $l$, which equals $1$ if $k=l$ and  $0$ if $k\neq l$.
By Lemma~\ref{Lem-killed process str Feller}, we know that for every $k \in\ss$, $\wdt{X}^{(k)}$ has the strong Feller property. Therefore, in view of
Proposition 6.1.1 in \cite{MeynT-93}, 
we
derive that $\wdt{P}^{(k)}(t,x,A)$ and every term in the series on
the right-hand side of \eqref{(FP22)}
are lower semicontinuous with respect to $x$
whenever $A$ is an open set in $\B(\R ^{d})$. Note that $\ss$ is a countably infinite set and has
discrete metric. Therefore it follows that
the left-hand side of  \eqref{(FP22)} is lower semicontinuous with
respect to $(x,k)$ for every $l \in \ss$ whenever $A$ is an open set
in $\B(\R ^{d})$.  Consequently, $(X,\La)$ has the strong Feller property
(see Proposition 6.1.1 in \cite{MeynT-93} again). The theorem is
proved.\qed \end{proof}

\begin{rem}{\rm
\cite{Shao-15} proves that for a state-independent regime-switching diffusion processes, the strong Feller property for each subdiffusion implies the strong Feller property for regime-switching diffusion processes. This work further proves this implication for state-dependent  regime-switching jump diffusion processes.
}\end{rem}

\begin{rem}{\rm The strong Feller property for regime-switching jump diffusions was also studied in \cite{XiZ-17}, where it is assumed that $\nu(U) < \infty$ is a finite measure, i.e.,  the jump part is modeled  by a compound Poisson process. In addition,   a finite-range condition for the switching component   is placed in that paper and is key to the analyses there. Here these two restrictions are removed. }
\end{rem}


\end{document}